%% file: Radin.tex
\documentclass{article}
\usepackage[utf8]{inputenc}
\usepackage{amsmath,amssymb,amsthm, xcolor, enumitem, comment, todonotes}
\usepackage{url}
\usepackage{hyperref}

\theoremstyle{definition}
\newtheorem{definition}{Definition}[section]

\newtheorem{observation}[definition]{Observation}
\newtheorem{question}[definition]{Question}

\newtheorem{notation}[definition]{Notation}
\newtheorem{claim}[definition]{Claim}

\newtheorem{remark}[definition]{Remark}

\theoremstyle{plain}
\newtheorem{theorem}[definition]{Theorem}
\newtheorem{proposition}[definition]{Proposition}
\newtheorem{lemma}[definition]{Lemma}
\newtheorem{corollary}[definition]{Corollary}

\input omermakros

\title{Compactness and Guessing Principles in the Radin Extensions}
\author{Omer Ben-Neria\footnote{The first author was partially supported by the Israel Science Foundation (Grant 1832/19).} \\  \href{omer.bn@mail.huji.ac.il}{omer.bn@mail.huji.ac.il} 
   \and Jing Zhang \footnote{
The second author was supported by the Foreign Postdoctoral Fellowship Program of the Israel Academy of Sciences and Humanities and by the Israel Science Foundation (grant
agreement 2066/18). }
 \\ \href{jingzhangjz13@gmail.com}{jingzhan@alumni.cmu.edu} }

\date{\today}

\begin{document}

\maketitle

\begin{abstract}
    We investigate the interaction between compactness principles and guessing principles in the Radin forcing extensions. In particular, we show that in any Radin forcing extension with respect to a measure sequence on $\kappa$, if $\kappa$ is weakly compact, then $\diamondsuit(\kappa)$ holds. This provides contrast with a well-known theorem of Woodin, who showed that in a certain Radin extension over a suitably prepared ground model relative to the existence of large cardinals, the diamond principle fails at a strongly inaccessible Mahlo cardinal. Refining the analysis of the Radin extensions, we consistently demonstrate a scenario where a compactness principle, stronger than the diagonal stationary reflection principle, holds yet the diamond principle fails at a strongly inaccessible cardinal, improving a result from \cite{MR3960897}.
\end{abstract}

\section{Introduction}

This paper contributes to the study of the interaction between compactness principles and guessing principles, specifically, in the context of Radin forcing \cite{MR670992}.
Recall that for a regular uncountable cardinal $\kappa$, $\diamondsuit(\kappa)$ asserts the existence of a sequence $\la S_\alpha\subset \alpha: \alpha<\kappa\ra$ such that for any $X\subset \kappa$, $\{\alpha<\kappa: X\cap \alpha = S_\alpha\}$ is stationary. An old open problem in this area asks if $\diamondsuit(\kappa)$ must hold at a weakly compact cardinal $\kappa$. 
We prove that this is indeed the case in the Radin forcing extensions, answering Question 37 from \cite{MR3960897} negatively.

\begin{theorem}\label{main}
Let $R_{\bar{U}}$ be the Radin forcing defined using a measure sequence $\bar{U}$ on $\kappa$. Then
\[
\Vdash_{R_{\bar{U}}} (\kappa \text{ is weakly compact } \implies \Diamond(\kappa) \text{ holds }).
\] 
\end{theorem}

The special attention given to Radin forcing in this context, originates in results of Woodin \cite{CummingsWoodin}, who used Radin forcing to establish the consistency of a large cardinal $\kappa$, such as strongly inaccessible, Mahlo, and greatly Mahlo, with $\neg\Diamond(\kappa)$. In fact, Radin forcing is the only known method for producing models where the diamond principle fails fully on any large cardinals. 

The history of the relation between the diamond and compactness principles, goes back to the work of Kunen and Jensen \cite{JensenKunen}, who showed that  $\Diamond(\kappa)$ must hold at every subtle cardinal. 
In fact, they prove that the stronger property $\Diamond(Reg_\kappa)$ holds at such cardinals, where $\Diamond(Reg_\kappa)$ asserts that there exists a diamond sequence supported on regulars. 
The consistency of  $\neg\Diamond(Reg_\kappa)$ on weak compact cardinals was shown by Woodin, and improved by Hauser \cite{MR1164732} to indescribable cardinals, and by D\v{z}amonja and
Hamkins \cite{MR2279655} to strongly unfoldable cardinals. 
Each of these consistency results concerning the failure of the diamond principle on the regulars is established from its minimal corresponding  large cardinal assumption. 
In contrast, $\neg\Diamond(\kappa)$ at relatively small large cardinals, such as Mahlo cardinals, is known to have a significantly stronger consistency strength. 
 Jensen \cite{Jensen69} has shown that $\neg\Diamond(\kappa)$ 
at a Mahlo cardinal $\kappa$ implies the existence of $0^\#$.  Zeman \cite{Zeman00} improved
the lower bound to the existence of an inner model with a cardinal $\kappa$, such that for every $\gamma < \kappa$, the set
$\{ \alpha < \kappa \mid o(\alpha) \geq \gamma\}$ is stationary in $\kappa$. 
The last large cardinal assumption is quite close to the hypermeasurability (large cardinal) assumptions used by Woodin to force $\neg\Diamond(\kappa)$ at a greatly Mahlo cardinal.

In \cite{MR3960897} the first author studied Radin forcing $R_{\bar{U}}$, and the connection between 
properties of measure sequence $\bar{U}$ on $\kappa$ and large cardinal properties of $\kappa$ in generic extensions by $R_{\bar{U}}$. It is shown that Woodin's construction of $\neg\Diamond(\kappa)$ can be extended to large cardinal properties such as stationary reflection principles. A question on whether we can extend the analysis to get a model of $\kappa$ being weakly compact and $\neg 
\diamondsuit(\kappa)$ was asked in \cite{MR3960897}. Theorem \ref{main} answers this question in the negative by showing that this approach cannot yield significantly stronger results. 
Two properties of measure sequences $\bar{U}$ which were isolated in   \cite{MR3960897}, are the Weak Repeat Property (WRP) and Local Repeat Property (LRP). It is shown in \cite{MR3960897} that $\bar{U}$ satisfies WRP if and only if $\kappa$ is weakly compact in generic extensions by $R_{\bar{U}}$, and asks if the stronger property of LRP has a similar characterization. 
We answer this question in Proposition \ref{LRPcharacterization}, showing that LRP characterizes the large cardinal property of almost ineffability. 
In the rest of the paper we extend the study of compactness and guessing principles in Radin forcing extensions. 
The organization of the paper is as follows: 

\begin{enumerate}
    \item In Section \ref{premilinaries}, we provide some background on the type of forcing notions that we will work with for the rest of the paper.
    \item In Section \ref{mildlargecardinals}, we study  characterizations of almost ineffable cardinals and weakly compact cardinals in the Radin extension based on certain properties of the measure sequence used to defined the forcing.
    \item In Section \ref{guessing}, we isolate scenarios when a variety of guessing principles can hold in the Radin extensions. 
    \item In section \ref{amenable}, we demonstrate a scenario where a compactness principle, stronger than the diagonal stationary reflection principle, holds but the diamond principle fails at a strongly inaccessible cardinal. 
\end{enumerate}

\section{Preliminaries}\label{premilinaries}

\subsection{Measure sequences}
A \emph{measure sequence} is a sequence $\bar{w} = \la \kappa(\bar{w})\ra \fr \la w(\tau) \mid \tau < lh(\bar{w})\ra$, where each $w(\tau)$ is a $\kappa(\bar{w})$-complete ultrafilter on $V_{\kappa(\bar{w})}$, and $\bar{w}$ is derived from an elementary embedding $j=j_{\bar{w}} : V \to M$ with $\crit(j) = \kappa(\bar{w})$ in the sense that for $A \subseteq V_{\kappa(\bar{w})}$ and $\tau < lh(\bar{w})$,
$$A \in w(\tau) \iff  \bar{w}\uhr\tau \in j(A).$$ 
In particular, $w(0)$ is equivalent to the normal measure derived from $j$ and $lh(\bar{w})\leq j(\kappa(\bar{w}))$.
The class of all measure sequences is denoted by $\MS$. 
For a set $A \subseteq \MS$, we denote the set of critical points $\kappa(\bar{w})$ for $ \bar{w} \in A$ by  $O(A)$. 

To clarify, for a measure sequence $\bar{w}$ and $\tau\leq lh(\bar{w})$, $\bar{w}\restriction \tau$ stands for $\la \kappa(\bar{w})\ra \fr \la w(\eta): \eta<\tau\ra$ and $\bigcap \bar{w}$ stands for $\bigcap_{\eta<lh(\bar{w})} w(\eta)$.\\

All measure sequences $\bar{w}$ in our constructions are assumed to satisfy $\MS \cap V_{\kappa(\bar{w})} \in \bigcap \bar{w}$ (see the discussion involving the set $\bar{A}$ in \cite[page 1402]{MR2768695}, for further details).

\noindent

\begin{definition}
Let $\Bar{U}$ be a measure sequence on $\kappa=\kappa(\bar{U})$ constructed by $j: V\to M$, and  $A\subset \mathcal{MS}\cap V_\kappa$. We say 
\begin{enumerate}
    \item $A$ is \emph{$\bar{U}$-measure-one} if $A\in \bigcap \bar{U}$. 
    \item $A$ is \emph{$\bar{U}$-tail-measure-one} if there is some $\gamma<lh(\bar{U})$ such that $A \in U(i)$ for all $\gamma<i<lh(\bar{U})$.
    \item $A$ is \emph{$\bar{U}$-positive} if $\{i<lh(\bar{U}): A\in U(i)\}$ is cofinal in $lh(\bar{U})$.
    \item $A$ is \emph{$\bar{U}$-stationarily-positive} if $\{i<lh(\bar{U}): A\in U(i)\}$ is a stationary subset of $lh(\bar{U})$.
    \item $A$ is \emph{$\bar{U}$-non-null} if there is some $\gamma<lh(\bar{U})$, $A\in U(\gamma)$.
\end{enumerate}

\end{definition}

\begin{definition}
We say a function $b : \MS \cap V_\kappa \to V_\kappa$
\begin{enumerate}
    \item is a \emph{measure function} if for any $\Bar{w}\in dom(b)$, $b(\Bar{w})\in \bigcap \Bar{w}$. 
    \item is a \emph{tail measure function} if for any $\Bar{w}\in dom(b)$, $b(\Bar{w})$ is $\bar{w}$-tail-measure-one.
\end{enumerate}
\end{definition}

\subsection{Radin forcing}

Let $R_{\Bar{U}}$ be the Radin forcing defined using $\Bar{U}$, constructed by $j: V\to M$. This forcing was first invented in \cite{MR670992}. The notations regarding the measure sequence and its constructing embedding are fixed for the rest of the paper unless otherwise stated.

Our notations and presentation follow \cite{MR3960897} or \cite{MR2768695} for the most part, with the exception of using the forcing convention by which for two conditions $p,q$, $p \geq q$ means that $q$ extends $p$ (i.e., $q$ is more informative). We refer the readers to the above for the definition of this forcing. 

\begin{notation}
\begin{enumerate}
    \item For any $A\in \bigcap \bar{U}$, there exists another $A'\subset A$ in $\bigcap \bar{U}$ satisfying that for any $\bar{w}\in A'$, $A'\cap V_{\kappa(\bar{w})}\in \bigcap \bar{w}$. We may without loss of generality assume all the measure one sets satisfy this property for the rest of this paper.
    \item 
    Conditions $p \in R_{\bar{U}}$ are finite sequences $ p = \la d_i \mid i \leq k\ra$ where each $d_i$ is either of the form $d_i = \la \kappa_i\ra$ for some $\kappa_i < \kappa$, or of the form $d_i = \la \bar{\mu}_i,a_i\ra$ where $\bar{\mu}_i$ is a measure sequence of length $lh(\bar{\mu}_i) > 0$ and $a_i \in \cap \bar{\mu}_i$. 
    We denote $\bar{\mu}(d_i) = \la \kappa_i\ra$ in the former case and $\bar{\mu}(d_i) = \bar{\mu}_i$ in the latter. We require that the top component $\bar{\mu}(d_k) = \bar{U}$.
    We also write $p = p_0\fr d_k$, where $p_0 = \la d_i \mid i < k\ra$ denotes the \emph{lower part} of $p$, and  $d_k = (\bar{U}, A^p)$ denotes its \emph{top part}. 
    \item Given a measure sequence $\bar{w}$ and a $\bar{w}$-measure-one set $c$, we say a finite sequence of measure sequences with increasing critical points $\overrightarrow{\eta}$ is \emph{addable} to $(\bar{w}, c)$, or $\overrightarrow{\eta}<<(\bar{w}, c)$, if for each $\bar{u}\in \overrightarrow{\eta}$, $\bar{u}\in c$ and $c\cap V_{\kappa(\bar{u})}\in \bigcap \bar{u}$.
\end{enumerate}
\end{notation}

\begin{remark}
We clarify the following abuse of notations:
$\fr$ could mean concatenation or one-step extension depending on the context. However, our usage is without ambiguity: 
\begin{enumerate}
\item If the object after $\fr$ is a pair, for example $p\fr (\bar{w}, A)$ where $\bar{w}$ is a measure sequence and $A$ is $\bar{w}$-measure-one, then $\fr$ means concatenation. In this case, $\kappa(\bar{w})>\kappa(\bar{u})$ for any measure sequence $\bar{u}$ appearing on $p$. 
\item If the object after $\fr$ is a measure sequence, for example $p\fr \bar{w}$, then this means it is a one-step extension. In this case, 
$\kappa(\bar{w})$ belongs to the measure one set $A^{p_i}$ in one of the components $p_i$ of $p$.
\end{enumerate}
\end{remark}

\begin{definition}
$\Bar{U}$ satisfies the \emph{Repeat Property} (RP) if there exists $\gamma<lh(\Bar{U})$ such that $\bigcap \Bar{U}\restriction \gamma = \bigcap \Bar{U}$. We say $\gamma$ is a repeat point for $\bar{U}$ and $\gamma$ witnesses $\Bar{U}$ satisfies RP. 
\end{definition}

\begin{definition}
We say $\gamma<lh(\Bar{U})$ is a $\emph{weak repeat point}$ if $\gamma$ witnesses $\Bar{U}\restriction \gamma+1$ satisfies RP. If $\gamma$ is not a weak repeat point, then it is \emph{novel}.
\end{definition}

%\begin{remark}
Mitchell \cite{Mitchell82} and Cummings-Woodin \cite{CummingsWoodin} independently proved that $\kappa$ is measurable in $V^{R_{\Bar{U}}}$ if and only if 
 $\cap\Bar{U} = \cap\bar{U'}$ for some measure sequence $\bar{U'}$ satisfying the RP.
%\end{remark}

For the rest of the paper, we may assume $\bar{U}$ does not satisfy RP. The reason is that suppose $\gamma$ is the first repeat point of $\bar{U}$, then $R_{\bar{U}}$ is forcing equivalent to $R_{\bar{U}\restriction \gamma}$ (as by definition the forcing only depends on $\bigcap \bar{U}=\bigcap\bar{U}\restriction \gamma$; see \cite{CummingsWoodinbook}) and $\bar{U}\restriction \gamma$ does not satisfy RP. Notice that if $\bar{U}$ does not satisfy RP, then there are unboundedly many $\gamma<lh(\bar{U})$ that are not weak repeat points.

\begin{remark}
If $\bar{U}$ does not satisfy RP and $cf(lh(\bar{U}))\leq \kappa$, then in $V^{R_{\bar{U}}}$, $\kappa$ becomes singular. This follows from the arguments in \cite{MR2768695}.
\end{remark}

\begin{remark}\label{tail2^kappa}
Suppose $|2^\kappa|^M$ does not divide $lh(\bar{U})$, then there exists some $\gamma<lh(\bar{U})$, such that for all $\gamma'\in (\gamma, lh(\bar{U}))$, $\gamma$ is not a weak repeat point. We sketch this when $lh(\bar{U})<|2^\kappa|^M$. Fix some function $f$ on $\kappa$ such that for each $\alpha<\kappa$, $f(\alpha)$ outputs a well ordering of $2^\alpha$ of length $|2^\alpha|$. For any $\gamma<lh(\bar{U})$, let $A\subset \kappa$ be such that $j(f)(\kappa)(\gamma)=A$. Consider the set $B_\gamma=\{\bar{w}\in V_\kappa\cap \MS: f(\kappa(\bar{w}))(lh(\bar{w}))=A\cap \kappa(\bar{w})\}$. It is easy to see that $B_\gamma\in U(\gamma)-\bigcup_{\gamma'<\gamma} U(\gamma')$. In general without assuming $lh(\bar{U})<|2^\kappa|^M$, by the assumption there exists some $\delta$ and $0<\rho<|2^\kappa|^M$ such that $lh(\bar{U})=\delta\cdot |2^\kappa|^M + \rho$. Then for $\delta\cdot |2^\kappa|^M \leq \gamma<lh(\bar{U})$ we modify the definition of $B_\gamma$ such that it contains $\bar{w}\in V_\kappa\cap \MS$ such that $f(\kappa(\bar{w}))(lh(\bar{w}) \mod |2^{\kappa(\bar{w})}|)=A\cap \kappa(\bar{w})$.
\end{remark}

\begin{comment}
By a theorem of Woodin (see \cite{CummingsWoodinbook}), for any measure sequence $\Bar{U}$, there exists another measure sequence $\Bar{U}'$ with no weak repeat point, such that $R_{\Bar{U}}$ is equivalent to $R_{\Bar{U}'}$.%\todo{Maybe need to change the definition of a measure sequence derived form an embedding $j$ for that, to allow skipping some of the measures.}
\end{comment}

\begin{definition}
For any generic $G\subset R_{\bar{U}}$ over $V$, 
let 
\[\MS_G = \{ \bar{w} \in \MS \mid \exists p \in G, \thinspace p=\la d_i: i\leq k\ra, \bar{w}=\bar{\mu}(d_i) \text{ for some }i<k \}.\]
\[ C_G = \{ \kappa(\bar{w}) \mid \bar{w} \in \MS_G\}\]
We say $A\subset \kappa$ is \emph{generated by a set in $V$} if there exists $B\subset \mathcal{MS}$ such that $A= O(B\cap \MS_G)=_{\mathrm{def}}\{\kappa(\bar{w}): \bar{w}\in B\cap \MS_G\}$. 
\end{definition}

%\todo{Clarify: does MSG only consist of limit points? Maybe need two definitions one only for limit points and the other for all}

The following theorem asserts that any club in the Radin extensions $V[G]$, $G \subseteq R_{\bar{U}}$, where the regularity of $\kappa$ is preserved, contains a club generated by a $\bar{U}$-tail-measure-one set in $V$. 

\begin{theorem}[\cite{MR3960897}]\label{tailclub}
If $\bar{U}$ satisfies $cf(lh(\bar{U}))\geq \kappa^+$, where $\kappa=\kappa(\bar{U})$, then given $p_0^\frown (\bar{U}, A)=p\Vdash \dot{\tau}$ is a club subset of $\kappa$, there exists a measure one set $A'\subset A$ and a $\bar{U}$-tail-measure-one set $\Gamma$ such that $p_0^\frown (\bar{U}, A')$ forces $O(\Gamma\cap G) \subset \dot{\tau}$. 
\end{theorem}

The converse is also true: 

\begin{theorem}[\cite{MR3960897}]\label{tailclubconverse}
Suppose $\bar{U}$ is a measure sequence not satisfying RP and $B\subset \MS$ is a $\bar{U}$-tail-measure-one set. Then in $V[G]$ where $G\subset R_{\bar{U}}$ is generic over $V$, $O(B\cap \MS_G)$ contains a club subset of $\kappa$.
\end{theorem}

The following result shows that Theorem \ref{tailclub} about clubs in Radin generic extensions does not extend to stationary sets.
\begin{proposition}
Suppose $lh(\bar{U})\leq |2^\kappa|^M$ and $cf(lh(\bar{U}))\geq \kappa^+$, then in $V^{R_{\bar{U}}}$, there exists a partition $c: \kappa \to \kappa$, such that for any unbounded $A\subset C_G$ generated by a set in $V$, it is the case that $c'' O(A)=\kappa$.
In particular, $O(B)$ is not $c$-homogeneous for any $\bar{U}$-non-null set $B$.
\end{proposition}

\begin{proof}
Let $(x^\alpha_l: l<|2^\alpha|)$ be an injective enumeration of $2^\alpha$ for $\alpha<\kappa$. Let $(x^\kappa_l: l<|2^\kappa|^M)=j(\la(x^\alpha_l: l<|2^\alpha|): \alpha<\kappa\ra)(\kappa)$. For $\alpha\leq \kappa$ and $\tau<(2^\alpha)^M$, define $A^{\alpha}_\tau =\{\bar{v}: x^{\kappa(\bar{v})}_{lh(\bar{v})}= x^{\alpha}_\tau\cap \kappa(\bar{v})\}$. Note that $A^{\kappa(\bar{w})}_{lh(\bar{w})}\not \in \bigcup \bar{w}$, for $\bar{U}$-measure-one many $\bar{w}$. To see this, it suffices to see that for any $\tau<lh(\bar{U})$, $A^\kappa_\tau \not\in \bigcup \bar{U}\restriction \tau$. Fix $\tau'<\tau$, we need to see that $A^\kappa_\tau \not \in U(\tau')$, which is equivalent to $\bar{U}\restriction \tau' \not\in j(A^\kappa_{\tau})$. Since $x^\kappa_{\tau'}\neq x^\kappa_\tau = j(x^\kappa_\tau)\cap \kappa$, by the definition of $A^\kappa_{\tau}$, we indeed have that $\bar{U}\restriction \tau' \not \in j(A^\kappa_\tau)$.

Define $A^{\alpha}_{\geq \tau}$ as follows: $\bar{v}\in A^{\alpha}_{\geq \tau}$ iff $\bar{v}\in A^{\alpha}_{\tau}$ or some initial segment of $\bar{v}$ is in $A^{\alpha}_{\tau}$. In $V[G]$ where $G\subset R_{\bar{U}}$ is generic over $V$, define a mapping $c': \MS_G\to \MS_G$ such that $\bar{w}$ is mapped by $c'$ to the maximal element in $\MS_G\cap A^{\kappa(\bar{w})}_{\geq lh(\bar{w})}$ if it exists. 

We show, using a density argument, that $c'$ has the following property: for any $A\in U(\gamma)$ and $B\in U(\tau)$ where $\gamma\leq \tau$, $c'[A \cap \MS_G] \cap B \neq \emptyset$.

Given $p=p_0 \fr (\bar{U}, E)\in R_{\bar{U}}$, we find $\bar{w}$ from $(E \cap B\cap A^{\kappa}_{\geq \gamma}) - V_{\kappa(\max(p_0))+1}$ and extend $p$ to $p\fr \bar{w}$. The reason why such $\bar{w}$ exists is that $A^{\kappa}_{\geq \gamma}\in \bigcap_{\gamma\leq i} U(i)$. Then we find some $\bar{u}\in A\cap E- V_{\kappa(\bar{w})+1}$ such that $A^{\kappa}_{\geq \gamma} \cap V_{\bar{u}} = A^{\kappa(\bar{u})}_{\geq lh(\bar{u})}$, and then we extend to $p_0\fr (\bar{w}, E\cap V_{\kappa(\bar{w})})\fr (\bar{u}, E-  A^{\kappa(\bar{u})}_{\geq lh(\bar{u})})\fr (\bar{U}, E)$, which forces $\dot{c}'(\bar{u})=\bar{w}$.

Finally, in $V$, if we let $*: \MS \to \kappa$ be such that for any $\alpha\in \kappa$, $(*)^{-1}(\alpha)$ is $\bar{U}$-positive, then in $V[G]$, $c=_{def} *\circ c'$ is as desired, namely, for any $A\subset \MS$ that is $\bar{U}$-non-null, $c [A\cap \MS_G] = \kappa$.
\end{proof}

The next lemma is a well-known fact that concerns getting nice representations for certain sets in the Radin forcing extension.

\begin{lemma}\label{RepresentationLemma0}
Suppose $p\in R_{\bar{U}}$ and a sequence of names $\la \dot{x}_\alpha\subset \alpha: \alpha<\kappa\ra$ are given. Then there exists an extension $q\leq p$ with $q_0=p_0$ and a function $f\in V$, such that for any $\bar{w}\in A^q$ , $f(\bar{w})$ is a $R_{\bar{w}}$-name for a subset of $\kappa(\bar{w})$ and $q\fr \bar{w} \Vdash f(\bar{w})=\dot{x}_{\kappa(\bar{w})}$.
\end{lemma}

\begin{proof}
For each $\bar{w}\in A^p$, we can find a $\bar{U}$-measure-one $A_{\bar{w}}\subset A^p$ and a $R_{\bar{w}}$-name $f(\bar{w})$ such that $p_0\fr (\bar{w}, A^p\cap V_{\kappa(\bar{w})})\fr (\bar{U}, A_{\bar{w}})$ forces $f(\bar{w})=\dot{x}_{\kappa(\bar{w})}$. To see this, consider $$D=_{def}\{t\in R_{\bar{w}}: t\leq_{R_{\bar{w}}}p_0\fr (\bar{w}, A^p\cap V_{\kappa(\bar{w})}), $$ $$\exists B_t\in \bigcap \bar{U}, \exists R_{\bar{w}}\text{-name }\dot{s}_t, t\fr (\bar{U}, B_t)\Vdash \dot{s}_t = \dot{x}_{\kappa(\bar{w})}\}.$$\
%todo{Do we need to go over different $t$ values? Is it possible to apply the factorization only with $p_0$? }
Observe that $D$ is a dense subset of $R_{\bar{w}}$ below $p_0\fr (\bar{w}, A^p\cap V_{\kappa(\bar{w})})$: given any $s\in R_{\bar{w}}$ extending $p_0\fr (\bar{w}, A^p\cap V_{\kappa(\bar{w})})$, since $R_{\bar{U}}/p\fr \bar{w}\simeq R_{\bar{w}}/p_0\fr (\bar{w}, A^p\cap V_{\kappa(\bar{w})}) \times R_{\bar{U}}/(\bar{U}, A^p-V_{\kappa(\bar{w})+1})$ and $(R_{\bar{U}}/(\bar{U}, A^p-V_{\kappa(\bar{w})+1}),\leq^*)$ is $(2^{|R_{\bar{w}}|})^+$-closed (recall that $\leq^*$ refers to the \emph{direct extension ordering}), we can find $t \leq_{R_{\bar{w}}} s$ and an $R_{\bar{w}}$-name $\dot{s}_t$ as well as some $B_t\in \bigcap \bar{U}$, such that $(\bar{U}, B_t)\Vdash_{R_{\bar{U}}} t\Vdash_{R_{\bar{w}}} \dot{x}_{\kappa(\bar{w})}=\dot{s}_t$. This shows that $D$ is dense in $R_{\bar{w}}$.

Let $D'\subset D$ be a maximal antichain below $p_0\fr (\bar{w}, A^p\cap V_{\kappa(\bar{w})})$. Then we can cook up a $R_{\bar{w}}$-name $f(\bar{w})$ for a subset of $\kappa(\bar{w})$ such that $t\Vdash f(\bar{w})=\dot{s}_t$ for any $t\in D'$. Now it is immediate that $f(\bar{w})$ and $A_{\bar{w}}=\bigcap_{t\in D} B_t$ satisfy the requirement. 

Finally, we let $q\leq p$ be such that $q_0=p_0$ and $A^q=\Delta_{\bar{w}\in A^p} A_{\bar{w}}$, which is the desired extension of $p$ as sought.
\end{proof}

An immediate consequence of the previous lemma is: 
\begin{corollary}\label{RepresentationLemma}
Let $G\subset R_{\bar{U}}$ be generic over $V$. Then in $V[G]$, for any $X\subset \kappa$, there exists $f\in V$ such that for every $\bar{w}\in dom(f)\cap \MS$, $f(\bar{w})$ is an $R_{\bar{w}}$-name for a subset of $\kappa(\bar{w})$ and there exists $\alpha<\kappa$, $(f(\bar{w}))^{G\upharpoonright R_{\bar{w}}} = X\cap \kappa(\bar{w})$ for each $\bar{w}\in \MS_G-V_\alpha$.
\end{corollary}

We finish this section with stating the following theorem.

\begin{theorem}[\cite{CummingsWoodinbook}]\label{approximation}
Let $G \subseteq R_{\bar{U}}$ be a $V$-generic filter. 
If $\kappa$ is regular in $V[G]$, then in $V[G]$, for every subset $X \subseteq \kappa$, 
if $X \cap \alpha \in V$ for all $\alpha < \kappa$ then $X \in V$. 
\end{theorem}

For a published account of the proof, the reader can consult \cite{MR3960897}.

\section{Mild large cardinals in Radin extensions}\label{mildlargecardinals}

In this section, we revisit two weakenings of the Repeat Property about the measure sequence $\bar{U}$, as considered in \cite{MR3960897}. In particular, each property corresponds to $\kappa$ being a certain large cardinal in the generic extension by $R_{\bar{U}}$.

\subsection{Almost ineffable cardinals}

\begin{definition}\label{LRP}
$\bar{U}$ satisfies the \emph{Local Repeat Property} (LRP) if for any measure function $b$, there exist $A\in \bigcap \Bar{U}$ and $\gamma<lh(\Bar{U})$ such that $j(b)(\Bar{U}\restriction \gamma)=A$.
\end{definition}
Since $\bar{U}$ does not contain a repeat point, it is not hard to check that if $\bar{U}$ satisfies LRP, then $cf(lh(\bar{U}))>\kappa$. To see this, suppose for the sake of contradiction that $cf(lh(\bar{U}))=\delta\leq \kappa$. Let $\langle \gamma_i: i<\delta\rangle$ be an increasing sequence cofinal in $lh(\bar{U})$. For each $i<\delta$, by the assumption that $\gamma_i$ is not a weak repeat point, there is some $A_i\in U(\gamma_i)-\bigcup_{\alpha<\gamma_i}U(\alpha)$. Define the following measure function $b$: for a measure sequence $\bar{w}$, if there exists some $i<\kappa(\bar{w})$ such that $A^c_i\cap V_{\kappa(\bar{w})}\in \bigcap \bar{w}$, then define $b(\bar{w})=A^c_i\cap V_{\kappa(\bar{w})}$ for the least $i$ as above. If there does not exist such $i$, then $b(\bar{w})=V_{\kappa(\bar{w})}$. By the assumption that $\bar{U}$ satisfies LRP, there exists some $\gamma<lh(\bar{U})$ and $A\in \bigcap \bar{U}$ such that $j(b)(\bar{U}\restriction \gamma)=A$. By the definition of $b$, $j(b)(\bar{U}\restriction \gamma)$ equals $A^c_i$ where $i$ is the least satisfying $\gamma_i\geq \gamma$. However, $A\neq A^c_i$ since $A\in U(\gamma_i)$ while $A^c_i\not\in U(\gamma_i)$, which is a contradiction.

\begin{remark}\label{arbitrarilylarge}
One obtains an equivalent definition by replacing ``...$\gamma<lh(\Bar{U})$ such that...'' with ``...cofinally many $\gamma<lh(\Bar{U})$ such that...'' in Definition \ref{LRP}.
\end{remark}

In \cite{MR3960897}, it was shown that if $lh(\Bar{U})<2^\kappa$, then $\Bar{U}$ does not satisfy LRP, thus making LRP incompatible with Woodin's argument for the failure of $\Diamond(\kappa)$ in $R_{\bar{U}}$-generic extension (see Theorem \ref{diamondfailsWoodin}).  
\cite{MR3960897} asks if LRP is equivalent to a natural large cardinal property of $\kappa$ in $R_{\bar{U}}$ generic extensions. We show next that this is the case. This shows that it is not an accident that LRP is incompatible with the argument for the failure of $\Diamond(\kappa)$.

\begin{definition}[{\cite{JensenKunen}}]
Fix a regular cardinal $\kappa$.
\begin{enumerate}
    \item $\kappa$ is \emph{almost ineffable} if for any sequence $\langle t_\alpha\subset \alpha: \alpha<\kappa\ra$, there is some $t\subset \kappa$ such that $\{\alpha: t\cap \alpha=t_\alpha\}$ is unbounded in $\kappa$.
    \item $\kappa$ is \emph{subtle} if for any sequence $\langle t_\alpha\subset \alpha: \alpha<\kappa\ra$ and any club $C\subset \kappa$, there are $\alpha<\beta\in C$ such that $t_\alpha \sqsubset t_\beta$.
\end{enumerate}
\end{definition}

\begin{remark}\label{clubequivalence}
It is easy to check that $\kappa$ is almost ineffable iff for any sequence $\langle t_\alpha\subset \alpha: \alpha<\kappa\ra$ and any club $C\subset \kappa$, there is some $t\subset \kappa$ such that $\{\alpha\in C: t\cap \alpha=t_\alpha\}$ is unbounded in $\kappa$. As a result, if $\kappa$ is almost ineffable, then $\kappa$ is subtle, which in turn implies $\diamondsuit(\kappa)$ by \cite{JensenKunen}.
\begin{comment}

To see this, we observe that $\kappa$ is almost ineffable iff for any $\langle t_\alpha\subset \alpha: \alpha<\kappa\ra$ and any club $C\subset \kappa$, there is some $t\subset \kappa$ such that $\{\alpha\in C: t\cap \alpha=t_\alpha\}$ is unbounded in $\kappa$. Given $\langle t_\alpha\subset \alpha: \alpha<\kappa\ra$ and a club $C\subset \kappa$ as above (we may assume $t_{\beta+1}=\{\beta\}$ for all $\beta$), for $\alpha\in \lim \kappa$, define \[t'_\alpha =  \begin{cases}
t_\alpha & \text{if }\alpha\in C \\
(\max C\cap \alpha) +1 & \text{otherwise}.
\end{cases}
\]
Apply almost ineffability to $\la t'_\alpha: \alpha<\kappa\ra$ so that we get an averaging set $t$. It is easy to check $\{\alpha\in C: t\cap \alpha =t_\alpha\}$ is unbounded in $\kappa$.
\end{comment}
\end{remark}

\begin{proposition}\label{LRPcharacterization}
$V^{R_{\Bar{U}}} \models  \kappa$ is almost ineffable iff $\Bar{U}$ satisfies \emph{LRP}. 
\end{proposition}

Proposition \ref{LRPcharacterization} follows from the following two lemmas.

\begin{lemma}
If $\bar{U}$ satisfies the LRP, then $\kappa$ is almost ineffable in $V^{R_{\bar{U}}}$
\end{lemma}
\begin{proof}
Let $p'\in R_{\bar{U}}$ and a name $\la \dot{s}_\alpha\subset \alpha: \alpha<\kappa\ra$ be given. By Lemma \ref{RepresentationLemma0}, we can get $p=p_0^\frown (\bar{U}, A) \leq^* p'$ and $\langle \dot{t}_{\bar{w}}: \bar{w}\in \mathcal{MS}\rangle$ such that $\dot{t}_{\bar{w}}$ is an $R_{\bar{w}}$-name and for any $\bar{w}\in A$, $p\fr \bar{w}\Vdash \dot{t}_{\bar{w}} = \dot{s}_{\kappa(\bar{w})}$. 

For each $\bar{w}\in A$, apply Lemma \ref{RepresentationLemma0} to $p_0\fr (\bar{w}, A\cap V_{\kappa(\bar{w})})$ and $R_{\bar{w}}$. As a result, we can find $b(\bar{w})\in \bigcap \bar{w}$, $b(\bar{w}) \subset A$, such that for any $\bar{v}\in b(\bar{w})$, it is the case that $p_0^\frown (\bar{v}, b(\bar{w})\cap V_{\kappa(\bar{v})})^\frown (\bar{w}, b(\bar{w}))\Vdash_{R_{\bar{w}}} \dot{t}_{\bar{w}}\cap \kappa(\bar{v})=h_{\bar{w}}(\bar{v})$, where $h_{\bar{w}}$ is a function defined on $\mathcal{MS}\cap V_{\kappa(\bar{w})}$ with $h_{\bar{w}}(\bar{v})$ being a $R_{\bar{v}}$-name for each $\bar{v}\in \dom (h_{\bar{w}})$. Let $\bar{h}$ denote $\langle h_{\bar{w}}: \bar{w}\in \mathcal{MS}\cap V_\kappa \ra$.

By the hypothesis that $\bar{U}$ satisfies LRP, we can find $\tau<lh(\bar{U})$ such that $B=j(b)(\bar{U}\restriction \tau)\in \bigcap \bar{U}$. Note $B\subset A$. We may assume $\tau\geq 2$ by Remark \ref{arbitrarilylarge}. Let $h=j(\bar{h})(\bar{U}\restriction \tau)$, then $h$ is a function from $\mathcal{MS}\cap V_\kappa$ to $V_\kappa$, taking $\bar{v}$ to $h(\bar{v})$, which is an $R_{\bar{v}}$-name for a subset of $\kappa(\bar{v})$.

Consider the set $C=\{\bar{w}\in B: lh(\bar{w})>1,B\cap V_{\kappa(\bar{w})}\subset b(\bar{w}), h\restriction V_{\kappa(\bar{w})}=h_{\bar{w}}\}\in U(\tau)$. Let $G$ be a generic filter for $R_{\bar{U}}$ over $V$ containing $p_0^\frown (\bar{U}, B)$ and in $V[G]$, we define a sequence of bounded subsets of $\kappa$ as follows: for each $\bar{w}\in C\cap \MS_G$, let $d_{\bar{w}}$ be the union of the interpretation $(h(\bar{v}))^G$ for all $\bar{v}\in \MS_G\cap V_{\kappa(\bar{w})}$.

\begin{claim}
For any $\bar{w}\in C\cap \MS_G$, $d_{\bar{w}}$ is well-defined, and it is equal to $t_{\bar{w}}=_{def}(\dot{t}_{\bar{w}})^G$.
\end{claim}

\begin{proof}[Proof of the Claim]
For any $\bar{v}<\bar{v}'$ both in $\MS_G\cap V_{\kappa(\bar{w})} \cap B - V_{rank(p_0)+1}$, we note that $$(h_{\bar{w}}(\bar{v}))^{G} , (h_{\bar{w}}(\bar{v}'))^{G} \sqsubset t_{\bar{w}}.$$ This is because the statement is forced by $p_0^\frown (\bar{v}, b(\bar{w})\cap V_{\kappa(\bar{v})})^\frown (\bar{w}, b(\bar{w}))$ (respectively, by $p_0^\frown (\bar{v}', b(\bar{w})\cap V_{\kappa(\bar{v}')})^\frown (\bar{w}, b(\bar{w}))$), extended by $p_0^\frown (\bar{v}, B\cap V_{\kappa(\bar{v})})^\frown (\bar{w}, B\cap V_{\kappa(\bar{w})})$ by the fact that $\bar{w}\in C$.
\end{proof}

\begin{claim}
For $\bar{w}<\bar{w}' \in \MS_G\cap C$, it is the case that $d_{\bar{w}}\sqsubset d_{\bar{w}'}$.
\end{claim}

\begin{proof}[Proof of the Claim]
Reason as above to check that for any $\bar{v}\in \MS_G\cap V_{\kappa(\bar{w})} \cap B - V_{rank(p_0)+1}$, $(h(\bar{v}))^G=(h_{\bar{w}}(\bar{v}))^G=(h_{\bar{w}'}(\bar{v}))^G$ is an initial segment of both $d_{\bar{w}}$ and $d_{\bar{w}'}$.
\end{proof}

Let $X=\bigcup_{\bar{w}\in \MS_G\cap C} d_{\bar{w}}$. Then we are done. Namely, in $V[G]$, $\{\alpha<\kappa: X\cap \alpha = (\dot{s}_\alpha)^G\}$ is unbounded in $\kappa$.

\end{proof}

\begin{lemma}
If $\kappa$ is almost ineffable in $V^{R_{\bar{U}}}$, then $\bar{U}$ satisfies LRP.
\end{lemma}

\begin{proof}
Fix a measure function $b$. Let $G\subset R_{\bar{U}}$ be generic over $V$. In $V[G]$, for each $\bar{w}\in G$, let $\gamma_{\bar{w}}<\kappa(\bar{w})$ be some ordinal such that $\MS_G-V_{\gamma_{\bar{w}}} \subset b(\bar{w})$. Consider the sequence $\bar{X}=\la X_{\bar{w}}=_{def}\{\gamma_{\bar{w}}\}\cup (b(\bar{w})-V_{\gamma_{\bar{w}}+1}): \bar{w}\in \MS_G\ra$. Note that $\bar{X}\subset V$. As $\kappa$ is almost ineffable, by Remark \ref{clubequivalence}, we can find $X$ such that $\{\bar{w}\in \MS_G: X\cap V_{\kappa(\bar{w})}=X_{\bar{w}}\}$ has size $\kappa$. In particular, $X$ is fresh. By Theorem \ref{approximation}, we know that $X\in V$. As a result, there is some $\gamma$, such that there are unboundedly mamy $\bar{w}$ with $\gamma_{\bar{w}}=\gamma$ and $X\cap V_{\kappa(\bar{w})}=\{\gamma_{\bar{w}}\}\cup (b(\bar{w})-V_{\gamma_{\bar{w}}+1})$. As a result, $X$ contains a  tail of $\MS_G$, which in turn implies $X\in \bigcap \bar{U}$.

Finally, we need to check that there is some $\tau<lh(\bar{U})$, such that $j(b)(\bar{U}\restriction \tau) =^* X$. More precisely, we mean here that the symmetric difference of two sets $j(b)(\bar{U}\restriction \tau)$ and $X$ is a subset of $V_\alpha$ for some $\alpha<\kappa$. This is as desired, since $X\in \bigcap \bar{U}$ implies that $j(b)(\bar{U}\restriction \gamma)\in \bigcap \bar{U}$.
But this immediately follows from the fact that in $V[G]$, $\{\kappa(\bar{w}): \bar{w}\in \MS_G \text{ and } X\cap V_{\kappa(\bar{w})}=^* b(\bar{w})\}$ is unbounded in $\kappa$.
\end{proof}

\subsection{Weakly compact cardinals}

\begin{definition}[\cite{MR3960897}, Definition 21]
We say $\Bar{U}$ satisfies the \emph{Weak Repeat Property} (WRP) if every measure function $b$ has a \emph{repeat filter} $W$ with respect to $\bar{U}$ in the following sense: 
\begin{enumerate}
    \item $W$ is a $\kappa$-complete filter extending the co-bounded filter on $\MS$, 
    \item any $X\in W$ is $\bar{U}$-non-null,
    \item $W$ \emph{measures} $b$, namely for each $\bar{u}$, if we let $X_{b, \bar{u}}=_{def} \{\bar{v}: \bar{u}\in b(\bar{v})\}$, then either $X_{b, \bar{u}}$ or $X_{b, \bar{u}}^c \in W$ and
    \item $[b]_{W} =_{def} \{\bar{u}: X_{b,\bar{u}}\in W\} \in \bigcap \bar{U}$.
\end{enumerate}
\end{definition}

In \cite{MR3960897}, it was shown that $\kappa$ is weakly compact in $V^{R_{\bar{U}}}$ iff $\bar{U}$ satisfies WRP. We supply more equivalent characterizations of weak compactness in Radin extensions, which will be useful later on.

\begin{proposition}\label{characterization}
The following are equivalent:
\begin{enumerate}
    \item $\Bar{U}$ satisfies WRP.
    \item $V^{R_{\Bar{U}}} \models \kappa$ is weakly compact.
    \item $V^{R_{\Bar{U}}} \models \kappa$ is inaccessible and the $V$-tree property holds. Namely, any $\kappa$-tree $T$, satisfying that $T\subset V$, which means exactly the following:
        \begin{itemize}
            \item the underlying set is a set of sequences, each of which is in $V$ and 
            \item is ordered by end-extension,
        \end{itemize}
        admits a cofinal branch. 
    \item for any measure function $b$, there exists $A\in \bigcap \Bar{U}$ such that for any $\alpha<\kappa$, there exists some $\beta<lh(\Bar{U})$, with $A\cap V_\alpha \subset j(b)(\Bar{U}\restriction \beta)$.
   
\end{enumerate}

\end{proposition}

\begin{proof}
The equivalence between 1 and 2 is proved in \cite{MR3960897} and it is immediate that 2 implies 3. 
\begin{itemize}

\item 3 implies 4: fix a measure function $b$. Let $G\subset R_{\bar{U}}$ be generic over $V$. In $V[G]$, the function $\bar{w}\in \MS_G \mapsto \gamma_{\bar{w}}$, where $\gamma_{\bar{w}}$ is the least $\gamma<\kappa(\bar{w})$ such that $\MS_G \cap V_{\kappa(\bar{w})}-V_{\gamma} \subset b(\bar{w})$, is regressive. Therefore, there exists some $\gamma$, such that $B=\{\bar{w}\in G: \gamma_{\bar{w}}=\gamma\}$ is stationary. More precisely, $O(B\cap \MS_G)$ is stationary in $\kappa$. Define a $\kappa$-tree $T$ such that $t\in T$ iff there exists some $\bar{w}\in B$ such that $t\sqsubseteq b(\bar{w})$ (namely there exists some $\alpha\leq \kappa(\bar{w})$ such that $t=b(\bar{w})\cap V_\alpha$). The order of $T$ is end extension. More precisely, $t\leq_{T} t'$ if there exists some $\alpha$ such that $t=t'\cap V_\alpha$. Notice that $T\subset V$. Apply the hypothesis, we can get a branch $d$ through $T$. By Theorem \ref{approximation}, $d\in V$. 

Observe that $d\in \bigcap \bar{U}$. To see this, it suffices to see $d$ contains $G-V_\gamma$. For any $\bar{u}\in G-V_{\gamma}$, $d\cap V_{\kappa(\bar{u})+lh(\bar{u})+1} \sqsubset b(\bar{w})$ for some $\bar{w}\in B$. By the choice of $\gamma$, we know that $\bar{u}\in b(\bar{w})$. Hence $\bar{u}\in d$.

Finally, it remains to show that for any $\alpha<\kappa$, $E=_{def}\{\bar{w}: d\cap V_\alpha \subset b(\bar{w})\}$ is $\bar{U}$-non-null. But this follows immediately from the fact that $O(E\cap \MS_G)$ is unbounded in $\kappa$.
\item 4 implies 1: given a measure function $b$, apply 4 to get $A\in \bigcap \bar{U}$ with the property. We first show a strengthening of the conclusion at 4 is possible. 
    \begin{claim}\label{4improved}
    We can find an $A$ as in the conclusion of 4 that satisfies for any $\alpha<\kappa$, there exists some $\beta<lh(\Bar{U})$, with $A\cap V_\alpha \sqsubset j(b)(\Bar{U}\restriction \beta)$.
    \end{claim}
    \begin{proof}[Proof of the Claim]
    Note the difference is ``$\subset$'' is replaced with ``$\sqsubset$''. Let $A'$ be given by the conclusion of 4. For each $\alpha<\kappa$, let $\beta_\alpha<lh(\bar{U})$ be the least witnessing ordinal that $A'\cap V_\alpha \subset j(b)(\bar{U}\restriction \beta_\alpha)$. Consider the tree on $V_\kappa$ defined by $T=\{j(b)(\bar{U}\restriction \beta_\alpha)\restriction V_\gamma: \gamma\leq \alpha<\kappa\}$, ordered by end extension. By the weak compactness of $\kappa$ in $V$, there exists a branch $A$ through the tree $T$. Clearly, for each $\alpha<\kappa$, there is some $\beta<lh(\bar{U})$, such that $A\cap V_\alpha \sqsubset j(b)(\bar{U}\restriction \beta)$ by the definition of the tree. Note also that $A\in \bigcap \bar{U}$, as evidenced by the fact that $A'\subset A$. To see the latter, given $\alpha<\kappa$, we can find some $\alpha'\geq \alpha$ such that $A\cap V_\alpha \sqsubset j(b)(\bar{U}\restriction \beta_{\alpha'})$. In particular, $A\cap V_\alpha \supset A'\cap V_\alpha$.
    \end{proof}

Define $W'\subset \MS$ such that $B\in W'$ iff there is some $\bar{u}\in A$ with $B=X_{b, \bar{u}}$ or there is some $\bar{u}\notin A$ with $B=X_{b, \bar{u}}^c$. Let $\mathcal{F}\subset \MS$ be the co-bounded filter. 
Let $W$ be the upward closure of $\{\bigcap_{B\in u} B: u\in [W'\cup \mathcal{F}]^{<\kappa}\}$. 
To check $W$ as defined is as desired, we verify that $W\subset \bigcup \bar{U}$ is a $\kappa$-complete filter. It suffices to show that for any $\langle B_i: i<\mu\ra \subset W'$ where $\mu<\kappa$, it is the case that $\bigcap_{i<\mu} B_i$ is $\bar{U}$-non-null, since $\mathcal{F}\subset \bigcap \bar{U}$. 
%Depend on how $B_i$ gets into $W'$, we can divide them into two groups with a function $g: \mu \to 2$ and elements $\la \bar{u}_i: i<\mu\ra$ such that $g(i)=0$ iff $B_i = X_{b, \bar{u}_i}$ where $\bar{u}_i\in A$ and $g(i)=1$ iff $B_i = X_{b, \bar{u}_i}^c$ where $\bar{u}_i\not\in A$. 
For each $i<\mu$, by the definition of $W'$, there is some $\bar{u}_i$ witnessing that $B_i\in W'$. Namely, either $\bar{u}_i\in A$ and $B_i=X_{b,\bar{u}_i}$ or $\bar{u}_i\not\in A$ and $B=X^c_{b, \bar{u}_i}$. 
Let $\alpha> \sup_{i<\mu} \kappa(\bar{u}_i)$. Claim \ref{4improved} implies that $\{\bar{v}: A\cap V_\alpha \sqsubset b(\bar{v})\}$, as a subset of $\bigcap_{i<\mu} B_i$, is $\bar{U}$-non-null.
        
\end{itemize}
\end{proof}

\begin{remark}
In item 4 above, we get an equivalent statement by replacing ``...there exists some $\beta<lh(\Bar{U})$...'' with ``...there exist unboundedly many $\beta<lh(\Bar{U})$...''.
\end{remark}

\begin{lemma}\label{notweaklycompact}
If $|2^\kappa|^M$ does not divide $lh(\Bar{U})$ as ordinals, then $V^{R_{\Bar{U}}} \models \kappa$ is not weakly compact. 
\end{lemma}
\begin{proof}
Assume first that $lh(\Bar{U})<|2^\kappa|^M$. We will indicate later how to deal with the general case.
Define $g$ such that $g(\alpha)=\langle x^\alpha_l: l<|2^\alpha|\ra$ is an injective enumeration of subsets of $\alpha$. Denote $j(g)(\kappa)=\la x^\kappa_l: l<|2^\kappa|^M\ra$. Let $A=x^\kappa_{lh(\bar{U})}$. In general for any $\beta<\kappa$ and any set $B\subset \beta$, denote $index_g^\beta(B)$ to be the unique $l$ such that $B=x^\beta_l$. Note for a fixed $\beta$ and $B$ as above, $index_g^\beta(B)$ only depends on $g(\beta)$. 

Define a measure function $b$ such that $$b(\bar{w})=\{\bar{u}\in V_{\kappa(\bar{w})}: lh(\bar{u})<index^{\kappa(\bar{u})}_g(x^{\kappa(\bar{w})}_{lh(\bar{w})}\cap \kappa(\bar{u}))<index_g^{\kappa(\bar{u})}(A\cap \kappa(\bar{u}))\}, $$$$ \text{whenever the set is in }\bigcap\bar{w} \text{ and otherwise }b(\bar{w})=V_{\kappa(\bar{w})}.$$

Let us call the second case in the definition of $b$ \emph{vacuous}. Observe that $\{\bar{w}\in V_\kappa: b(\bar{w}) \text{ is not vacuously defined}\}\in \bigcap\bar{U}$. To see this, let $E=_{\mathrm{def}}\{\bar{u}\in V_{\kappa}: lh(\bar{u})<index^{\kappa(\bar{u})}_{g} (x^\kappa_\tau \cap \kappa(\bar{u})) < index_{j(g)}^{\kappa(\bar{u})}(j(A)\cap \kappa(\bar{u}))= index_{g}^{\kappa(\bar{u})}(A\cap \kappa(\bar{u}))\}$. We show that $E\in \bigcap \bar{U}\restriction \tau$, which implies $j(b)(\bar{U}\restriction \tau)=E$. Given $\gamma<\tau$, we know that $\bar{U}\restriction \gamma\in j(E)$ iff $x^\kappa_\tau = x^\kappa_\beta$, for some $\beta$ satisfying $\gamma<\beta<index_{j(g)}^\kappa(A)=lh(\bar{U})$. The latter is true as evidenced by $\beta=\tau$.

Suppose for the sake of contradiction that $\kappa$ is weakly compact in $V^{R_{\bar{U}}}$. Apply 4 in Proposition \ref{characterization} to $b$ to get $B\in \bigcap \bar{U}$ such that:

\emph{for any $\alpha<\kappa$, there is $\beta<lh(\bar{U})$ with $B\cap V_\alpha \subset j(b)(\bar{U}\restriction \beta)$.} 
\newline
Apply $j$ again to the statement above, we get (in $M$): 

\emph{for any $\alpha<j(\kappa)$, there is $\beta<lh(j(\bar{U}))$ with $j(B)\cap V_\alpha \subset j(j(b))(j(\bar{U})\restriction \beta)$.} In particular, working in $M$, if we consider $\alpha=(2^\kappa)^{+M}<j(\kappa)$, then we can find some $\tau<j(lh(\bar{U}))$ such that $j(B)\cap (V_{\alpha})^M \subset j(j(b))(j(\bar{U})\restriction\tau)=$ 
\begin{equation}\label{jjkey}
\{\bar{u}\in j(V_\kappa): x^{j(\kappa)}_\tau\cap \kappa(\bar{u})= x^{\kappa(\bar{u})}_\beta, lh(\bar{u})<\beta<index_{j(j(g))}^{\kappa(\bar{u})}(j(j(A))\cap \kappa(\bar{u}))\}.
\end{equation}

For the sake of consistency, we use $\la x^{j(\kappa)}_l: l<dom[j(j(g))(j(\kappa))]\ra$ to denote the enumeration of $j(j(g))(j(\kappa))$. Since $lh(\bar{U})<dom(j(g)(\kappa))$, the elementarity of $j$ implies that $j(lh(\bar{U}))<dom(j(j(g))(j(\kappa)))$. Thus the sentence above makes sense.

%Since $j(B)\cap (V_{\alpha})^M\supset \{\bar{U}\restriction \xi: \xi<lh(\bar{U})\}$, 
We show first that there exists some $\gamma<lh(\bar{U})$ such that $x^{j(\kappa)}_\tau \cap \kappa = x^\kappa_\gamma$. To see this, fix some $\xi<lh(\bar{U})$. Since $\bar{U}\restriction \xi \in j(B)\cap (V_{\alpha})^M$, by (\ref{jjkey}), we know $x^{j(\kappa)}_\tau \cap \kappa = x^{\kappa}_\gamma$, for some $\gamma>\xi$ and $\gamma<index_{j(j(g))}^\kappa(j(j(A))\cap \kappa)=index_{j(j(g))}^\kappa(A)=index_{j(j(g)\restriction \kappa)}^\kappa(A)=index_{j(g)}^\kappa(A)=lh(\bar{U})$.

Finally, fix some $\xi'>\gamma$ and $\xi'<lh(\bar{U})$ such that $\bar{U}\restriction \xi' \in j(B)\cap (V_\alpha)^M$. Reasoning as above, we can find $\gamma'>\xi'>\gamma$ with $\gamma'<lh(\bar{U})$ such that $x^{j(\kappa)}_\tau \cap \kappa= x^\kappa_{\gamma'}$. But $x^\kappa_\gamma\neq x^\kappa_{\gamma'}$ by the injectivity of the enumeration. This is a contradiction.

In general without assuming $lh(\bar{U})<|2^\kappa|^M$, we modify as follows. First notice that $lh(\bar{U})\leq j(\kappa)$ and $j(\kappa)$ is divisible by $|2^\kappa|^M$. Hence, we may assume $lh(\bar{U})<j(\kappa)$. By Remark \ref{tail2^kappa}, there exists some $\mu < lh(\bar{U})$ so that $lh(\bar{U})=\mu + \delta$ with  $\delta<|2^\kappa|^M$, and $U(\mu)$ is novel. It follows that there is some $D \subseteq \MS \cap V_\kappa$ so that  $D \in U(i)$ if and only if $i < \mu$. 
For each $\bar{w} \in \MS \cap V_\kappa$ let 
$i_D(\bar{w})$ be least value $i < lh(\bar{w})$ so that $D \cap V_{\kappa(\bar{w})} \not\in \bar{w}(i)$ if such $i$ exists, and leave $i_D(\bar{w})$ undefined otherwise.
If $i_D(\bar{w})$ is defined then clearly $i_D(\bar{w}) < lh(\bar{w})$, and we set $\delta_D(\bar{w})$ be the unique $\delta$ for which 
$i_D(\bar{w}) + \delta = lh(\bar{w})$. 

%and $\mu$ is a multiple of $|2^\kappa|^M$. Note $\delta>0$ and is of cofinality at least $\kappa^+$. here is some $\nu$, $\mu \leq \nu < \mu + \delta$ so that $U(\mu)$ is novel, and we can take some $D\subset \MS\cap V_\kappa $ such that for any $i$, $D\in U(i)$ iff $i\leq \nu$.
Define $g$ as before and $A=x^\kappa_{\delta}$. Define $b$ such that 
$$b(\bar{w})=(D\cap V_{\kappa(\bar{w})})\cup  \{\bar{u}\in V_{\kappa(\bar{w})}\cap D^c: 
x^{\kappa(\bar{w})}_{\delta_D(\bar{w})} \cap \kappa(\bar{u}) = x^{\kappa(\bar{u})}_\beta
$$

%x^{\kappa(\bar{w})}_{lh(\bar{w})\mod |2^{\kappa(\bar{w})}|}\cap \kappa(\bar{u})=x^{\kappa(\bar{u})}_\beta,$$

$$\text{ where } \delta_D(\bar{u}) <\beta<index_g^{\kappa(\bar{u})}(A\cap \kappa(\bar{u}))\}, $$

$$ \text{whenever the set is in }\bigcap\bar{w} \text{ and otherwise }b(\bar{w})=V_{\kappa(\bar{w})}.$$

Again observe that $\{\bar{w}\in V_\kappa: b(\bar{w}) \text{ is not vacuously defined}\}\in \bigcap\bar{U}$. Apply 4 in Proposition \ref{characterization} to $b$ to get $B\in \bigcap \bar{U}$ as before. If we let $\alpha=(lh(\bar{U})+2^\kappa)^{+M}$, then we can find some $\tau<j(lh(\bar{U}))$ such that $j(B)\cap (V_{\alpha})^M \subset j(j(b))(j(\bar{U})\restriction \tau)$. The choice of $\alpha$ ensures that $\{\bar{U}\restriction \nu: \nu<lh(\bar{U})\}\subset (V_\alpha)^M$.
Fix some $\xi<lh(\bar{U})$ such that $\xi>\mu$. Since $\bar{U}\restriction \xi \in j(B)\cap (V_{\alpha})^M$, it follows that $\bar{U}\restriction \xi \in j(j(b))(j(\bar{U})\restriction \tau)$. Since $j(j(D))\cap j(V_\kappa)=j(j(D)\cap V_\kappa)=j(D)$, we know that $\bar{U}\restriction \xi \not\in j(j(D))$ since $\bar{U}\restriction \xi\not\in j(D)$ and $\bar{U}\restriction \xi \in j(V_\kappa)$. As a result, by looking at the definition of $b$, we must have 
$$ x^{j(\kappa)}_{{j(j(\delta_D))(j(\bar{U})\uhr\tau) }} \cap \kappa = x^\kappa_\gamma, \text{ where }j(\delta_D)(\bar{U}\uhr \xi)=\xi-\mu <\gamma < \delta.$$ 

Fix another $\xi'<lh(\bar{U})$ such that $\xi' -\mu > \gamma$. Repeating the argument as in the case where $lh(\bar{U})<|2^\kappa|^M$, we get some $\gamma' >\gamma$ such that $x^{j(\kappa)}_{{j(j(\delta_D))(j(\bar{U})\uhr\tau) }} \cap \kappa = x^\kappa_{\gamma'}$. This contradicts with the fact that $x^\kappa_\gamma \neq x^\kappa_{\gamma'}$.
\end{proof}

\begin{remark}
The proof in Lemma \ref{notweaklycompact} actually produces a measure function $b$ such that there is \emph{no $\bar{U}$-positive} $A$, such that for any $\alpha<\kappa$, there is $\beta<lh(\bar{U})$, $A\cap V_\alpha \subset j(b)(\bar{U}\restriction \beta)$.
\end{remark}

\begin{remark}
Similar proof as in Lemma \ref{notweaklycompact} shows $\kappa$ is not weakly compact in $V^{R_{\bar{U}}}$ if there is some $\nu$ such that $lh(\bar{U})= \nu + |2^\kappa|^M$.
\end{remark}

\section{Guessing principles in Radin extensions}\label{guessing}

\begin{definition}
Let $V \subseteq W$ be two transitive models of set theory, with $\kappa$ being a regular cardinal in $W$ and $S\subset \kappa$ being a stationary subset. We say 
\begin{itemize}
    \item $V$-$\diamondsuit(S)$ holds (in $W$) if there exists $\langle s_\alpha\subset \alpha: \alpha<\kappa\rangle$ such that for all $X\subset \kappa$ in $V$, it is the case that $\{\alpha\in S: X\cap \alpha = s_\alpha\}$ is stationary.
    \item $V$-$\diamondsuit^-(S)$ holds (in $W$) if there exists $\langle \mathcal{S}_\alpha\in [\mathcal{P}(\alpha)]^{\leq |\alpha|}: \alpha<\kappa\rangle$ such that for all $X\subset \kappa$ in $V$, it is the case that $\{\alpha\in S: X\cap \alpha \in \mathcal{S}_\alpha\}$ is stationary.
\end{itemize}

\end{definition}

Cummings and Magidor \cite{CummingsMagidor} showed  that $V$-diamonds are equivalent to diamonds at $\kappa$ in Radin generic extensions of $V$ preserving the regularity of $\kappa$. 
%For completeness, we supply a proof in the following specific case. 

\begin{theorem}[Cummings-Magidor, \cite{CummingsMagidor}]\label{V-diamond}
Let $G\subset R_{\bar{U}}$ be generic over $V$ and assume that $\kappa$ is regular in $V[G]$. Then in $V[G]$, for any stationary $S\subset \kappa$, $V$-$\diamondsuit^-(S)$ is equivalent to $\diamondsuit^-(S)$.
\end{theorem}

\begin{observation}\label{GroundWitness}
Suppose $cf(lh(\bar{U}))\geq \kappa^+$. Then $V^{R_{\bar{U}}}\models \diamondsuit(\kappa)$ holds iff there exists a function $f\in V$ such that $\dom (f)=\MS \cap V_\kappa$ and $f(\bar{w})\in [2^{\kappa(\bar{w})}]^{\leq \kappa(\bar{w})}$, satisfying that for any $X\subset \kappa$, $\{\bar{w}: X\cap \kappa(\bar{w})\in f(\bar{w})\}$ is $\bar{U}$-positive. 
\end{observation}
\begin{proof}
A well-known result of Kunen (Theorem III.7.8 in \cite{MR2905394}) states that $\diamondsuit^-(S)$ holds iff $\diamondsuit(S)$ holds for any regular $\kappa$ and any stationary set $S\subset \kappa$.
Along with Theorem \ref{V-diamond}, we know $V^{R_{\bar{U}}}\models \diamondsuit(\kappa)$ iff $V^{R_{\bar{U}}}\models \diamondsuit^-(\kappa)$ iff $V^{R_{\bar{U}}}\models V$-$\diamondsuit^-(\kappa)$ iff 
$V^{R_{\bar{U}}}\models V$-$\diamondsuit(\kappa)$.
\begin{itemize}
    \item Suppose there exists some $p\in R_{\bar{U}}$ forcing that $\langle \dot{s}_\alpha: \alpha<\kappa\rangle$ is a $V$-$\diamondsuit(\kappa)$-sequence. By Lemma \ref{RepresentationLemma0}, we may find $q\leq p$ and $g$ such that for any $\bar{w}\in A^q$, $g(\bar{w})$ is an $R_{\bar{w}}$-name for a subset of $\kappa(\bar{w})$ in $V$ and $q\fr \bar{w}\Vdash g(\bar{w})=\dot{s}_{\kappa(\bar{w})}$. Since $R_{\bar{w}}$ is $\kappa(\bar{w})^+$-c.c, we can find $f(\bar{w})\in [2^{\kappa(\bar{w})}]^{\leq \kappa(\bar{w})}$ such that $1\Vdash_{R_{\bar{w}}} g(\bar{w})\in f(\bar{w})$. Suppose $X\subset \kappa$, then we know that we can find some $r\leq q$ such that $r \Vdash `` \{\kappa(\bar{w}): \bar{w}\in \dot{G}, X\cap \kappa(\bar{w})=g(\bar{w})\in f(\bar{w})\}$ is stationary''. This implies that $\{\bar{w}: X\cap \kappa(\bar{w})\in f(\bar{w})\}$ is $\bar{U}$-positive by Theorem \ref{tailclubconverse}.
    \item For the other direction, suppose there is one such $f$, then it is easy to verify that if $G\subset R_{\bar{U}}$ is generic over $V$, then $\{f(\bar{w}): \bar{w}\in G\}$ is the desired $V$-$\diamondsuit^-(\kappa)$-sequence. 
\end{itemize}

\end{proof}

The last observation allows us to reproduce the following result of Cummings and Magidor. More precisely, if $|2^\kappa|^M$ divides $lh(\bar{U})$, it suffices to find a function $f$ as in Observation \ref{GroundWitness}. For each $\beta<\kappa$, we let $\langle x^\beta_l: l<2^\beta\rangle$ enumerate all subsets of $\beta$. It is not hard to verify that $\bar{w}\mapsto \{x^{\kappa(\bar{w})}_{lh(\bar{w}) \  \mathrm{ mod }\ 2^{\kappa(\bar{w})}}\}$ is the function as sought.

\begin{theorem}[Cummings-Magidor, \cite{CummingsMagidor}]\label{diamondholds}
If $|2^\kappa|^M$ divides $lh(\Bar{U})$, then $V^{R_{\Bar{U}}} \models \diamondsuit(\kappa)$.
\end{theorem}
\begin{comment}
\begin{proof}Let $g$ be a function on $\kappa$ such that $g(\alpha)$ enumerates subsets of $2^\alpha$ in length $|2^\alpha|$ in a way that any $\sigma\subset \alpha$, $g(\alpha)(l)=\sigma$ for unboundedly many $l<|2^\alpha|$. Let $f(\bar{w})=g(\kappa(\bar{w}))(lh(\bar{w}) \mod |2^{\kappa(\bar{w})}|)$. Then it is easily checked that for any $X\subset \kappa$, there are unboundedly many $\xi<lh(\bar{U})$ such that $j(f)(\bar{U}\restriction \xi)=j(g)(\kappa)(\xi \mod |2^\kappa|^M)=X$. Hence the conclusion follows from Observation \ref{GroundWitness}. \end{proof}
\end{comment}
We are ready to prove the main result of the paper.
\begin{proof}[Proof of Theorem \ref{main}]
Let $\bar{U}$ be a measure sequence on $\kappa$ constructed by $j: V\to M$ and suppose $V^{R_{\bar{U}}}\models \kappa$ is weakly compact. By Lemma \ref{notweaklycompact}, $|2^\kappa|^M$ divides $lh(\bar{U})$. Hence, Theorem \ref{diamondholds} implies $\diamondsuit(\kappa)$ holds in $V^{R_{\bar{U}}}$.
\end{proof}

The following is a well-known theorem of Woodin (see \cite{CummingsWoodin} or
\cite{MR3960897} for published proofs). 

\begin{theorem}[Woodin]\label{diamondfailsWoodin}
If $|2^\kappa|$ does not divide $lh(\Bar{U})$ and $cf(lh(\bar{U}))\geq \kappa^+$, then $V^{R_{\Bar{U}}} \models \neg \diamondsuit(\kappa)$.
\end{theorem}

\begin{remark}
If $j: V\to M$ constructs $\Bar{U}$, and if $\Bar{U}\in M$, then $V^{R_{\Bar{U}}}\models \neg \diamondsuit(\kappa)$ whenever $|2^\kappa|^M $ does not divide $lh(\Bar{U})$. This is done by apply Woodin's theorem in $M$.
\end{remark}

Under certain circumstances, we can get stronger guessing principles. Recall the following characterization of stationary sets in the Radin extensions.

\begin{theorem}[\cite{MR3960897}]\label{stationary}
If $\dot{S}$ is an $R_{\bar{U}}$-name for a stationary subset of $\kappa$ and $p\in R_{\bar{U}}$, then there exists an extension $e=e_0\fr (\bar{U}, B)\leq p$ and a measure function $b$ such that 
\begin{enumerate}
    \item $Z_{e_0}=_{def}\{\bar{u}: \exists A\in \bigcap\bar{U}, e_0\fr (\bar{u}, b(\bar{u}))\fr (\bar{U}, A)\Vdash \kappa(\bar{u})\in \dot{S}\}$ is $\bar{U}$-positive, 
    \item for any $\overrightarrow{\eta}\in \MS^{<\omega}$ and $\overrightarrow{\eta}\subset B$, $Z_{\bar{e}} \downharpoonright \overrightarrow{\eta} =_{def} \{\bar{u}\in Z_{e_0}: \overrightarrow{\eta} << (\bar{u}, b(\bar{u}))\}$ is $\bar{U}$-positive. 
\end{enumerate}

\end{theorem}

\begin{proposition}\label{diamondstat}

If $\bar{U}$ is a measure sequence constructed by $j: V\to M$ and $M\models cf(lh(\bar{U})) \geq 2^\kappa$, then in $V^{R_{\bar{U}}}$, $\diamondsuit(S)$ holds for all stationary $S\subset \kappa$.

\end{proposition}

\begin{proof}

Let $\dot{S}$ be a name for a stationary subset of $\kappa$ and $p=p_0^\frown (\bar{U}, A)$ be a condition forcing this, which furthermore satisfies the conclusion of Theorem \ref{stationary}. 
We will construct a $V$-$\diamondsuit^-(S)$-sequence in $V^{R_{\bar{U}}}$. By Lemma \ref{V-diamond} and a theorem of Kunen, we know $\diamondsuit(S)$ holds in $V^{R_{\bar{U}}}$.

For each $\alpha$, we fix $\langle x^\alpha_l: l<|2^\alpha|\rangle$ an enumeration of subsets of $\alpha$ with cofinal repetitions. For each $Z\subset \mathcal{MS}$, we define $f_Z(\bar{w})=otp(\{\xi<lh(\bar{w}): \bar{w}\restriction \xi \in Z\})$.

In $V[G]$, we define a $V$-$\diamondsuit^-(S)$-sequence as follows: if $\bar{w}\in \MS_G$ and $\kappa(\bar{w})\in S$, then define $\mathcal{S}_{\kappa(\bar{w})}$ to consist of $x^{\kappa(\bar{w})}_l$ for those $l$ such that there is some $\overrightarrow{\eta}\in A^{<\omega}\cap V_{\kappa(\bar{w})}$ such that $\bar{w}\in Z_{p_0} \downharpoonright\overrightarrow{\eta}$ and $l=(f_{Z_{p_0} \downharpoonright\overrightarrow{\eta}}(\bar{w}) \mod 2^{\kappa(\bar{w})} )< 2^{\kappa(\bar{w})}$. Note that by the definition for each $\bar{w}\in \MS_G$ with $\alpha=\kappa(\bar{w})$, it is true that $|\mathcal{S}_\alpha|\leq |\alpha|$.

Let's check that this is as desired. Suppose not, then there exists an extension $q$ of $p$ and $X\subset \kappa$ and a tail measure one $\Gamma$ as witnessed by $\nu<lh(\bar{U})$ such that $q$ forces that $X$ is not guessed at any point at $\dot{S}\cap O(
\MS_G\cap \Gamma)$. 

Write $q$ as ${p_0'} ^\frown (\overrightarrow{\eta}, \overrightarrow{A}_{\overrightarrow{\eta}})^\frown (\bar{U}, B')$, where there is some $\bar{v}$ with $p_0,p_0'\in R_{\bar{v}}$ and $p_0'\leq_{R_{\bar{v}}} p_0$. Here $\overrightarrow{\eta}$ is an increasing sequence of measure sequence and $\overrightarrow{A}_{\overrightarrow{\eta}}$ is a sequence of measure one sets with respect to the measure sequences on $\overrightarrow{\eta}$.

By the choice of $p$, we know that $Z_{p_0} \downharpoonright\overrightarrow{\eta}$ is $\bar{U}$-positive. Let $l<2^\kappa$ be such that $x^\kappa_l = X$ and $j(f_{Z_{p_0} \downharpoonright\overrightarrow{\eta}})(\bar{U}\restriction \xi)=l$ ($\mathrm{mod}\ 2^\kappa$) for some $\bar{U}\restriction \xi\in j(\Gamma \cap Z_{p_0} \downharpoonright\overrightarrow{\eta})$.
By elementarity, there is some $\bar{u}$ reflecting $\bar{U}\restriction \xi$, namely, there is some $\bar{u}\in \Gamma\cap Z_{p_0} \downharpoonright\overrightarrow{\eta}$ such that if $l=f_{Z_{p_0} \downharpoonright\overrightarrow{\eta}}(\bar{u}) (\mathrm{mod}\  2^{\kappa(\bar{u})})$, then $x^{\kappa(\bar{u})}_l = X\cap \kappa(\bar{u})$. Since $\overrightarrow{\eta}<< (\bar{u}, b(\bar{u}))$, we know that ${p_0'} ^\frown (\overrightarrow{\eta}, \overrightarrow{A}_{\overrightarrow{\eta}})^\frown (\bar{u}, b(\bar{u})) ^\frown (\bar{U}, B')$ is compatible with $p_0^\frown (\bar{u}, b(\bar{u}))^\frown (\bar{U}, A)$. Hence a common extension will force $\kappa(\bar{u})\in \dot{S}$, as well as $\bar{u}\in \Gamma$ and $X\cap \kappa(\bar{u}) \in \dot{\mathcal{S}}_\alpha$, contradicting with our assumption.
\end{proof}

\begin{remark}
Proposition \ref{diamondstat} is optimal in the following sense: if $M\models cf(lh(\bar{U}))<2^\kappa$ and $\bar{U}\in M$, then in $V^{R_{\bar{U}}}$ there exists a stationary set $S$ such that $\diamondsuit(S)$ fails. It suffices to show there exists a $\bar{U}$-positive set $B$ such that $M\models ``\{\tau: B\in U(\tau)\}$ has size $<2^\kappa$''. Then we can finish with the argument of Theorem \ref{diamondfailsWoodin}.

Recall $j: V\to M$ constructs $\bar{U}$. By a theorem by Cummings-Woodin \cite{CummingsWoodinbook}, we may assume that $M=\{j(f)(\bar{U}\restriction \xi): \xi \text{ is novel}\}$. The point is that any measure sequence is equivalent, in the sense of the Radin forcing defined from it, to another one constructed by the embedding formed by taking the limit ultrapower along the novel points on the measure sequence. 
\begin{comment}
Let $N$ be the transitive collapse of $$Hull^M(j(V)\cup \{\bar{U}\restriction \xi: \xi<lh(\bar{U})\}).$$ Then $N$ is isormorphic to the direct limit of $\langle M_\xi \simeq Ult(V, U(\xi)), \pi_{\xi, \zeta}: \xi\leq \zeta<lh(\bar{U}) \rangle$, where for any $\xi<\zeta<lh(\bar{U})$, by the hypothesis that $\bar{U}$ does not contain a weak repeat point, there is some $A_{\xi}\in U(\xi)-\bigcup \bar{U}(<\xi)$, then $g_{\xi, \zeta}(\bar{w})$ returns the least $\gamma$ such that $\bar{w}\restriction \gamma\in A_\xi$. Then $\pi_{\xi, \zeta}: M_\xi \to M_{\zeta}$ maps $[f]_{U(\xi)}$ to $[f\circ g_{\xi,\zeta}]_{U(\zeta)}$. Let $\pi=\pi_{0, \infty}: V\to N$ and $\pi_{\xi, \infty}: M_\xi \to N$ be the respective limit embeddings. Let $k: N \to M$ be the elementary embedding defined such that $k(\pi_{\xi, \infty} ([f]))= j(f)(\bar{U}\restriction \xi)$.

Note that the critical point of $N$ should be at least $(2^\kappa)^+$ as computed in $N$, which is greater than $lh(\bar{U})$. In other words, $lh(\bar{U})$ has cofinality $<2^\kappa$ in $N$ as well. For simplicity, we may assume $M=N$ and $i=j$ for the discussion below. 
\end{comment}

By the simplification, we can find $f$ and a novel $\xi<lh(\bar{U})$ such that $j(f)(\bar{U}\restriction \xi)= E $, which is cofinal in $lh(\bar{U})$ and of size $<2^\kappa$, as computed in $M$. Let $A\in U(\xi)-\bigcup \bar{U}(<\xi)$. 
Consider $B=\{\bar{u}: \text{if $\xi$ is the least such that }\bar{u}\restriction \xi\in A, lh(\bar{u}) \in f(\bar{u}\restriction \xi)\}$.

We claim for $\tau>\xi$,  $B\in U(\tau)$ iff $\tau\in E$. If $\tau>\xi$, and $\bar{U}\restriction \tau \in j(B)$, then by the definition, $\xi$ is the least such that $\bar{U}\restriction \xi \in j(A)$, then $\tau$ has to be in $j(f)(\bar{U}\restriction \xi)=E$. Conversely, if $\bar{U}\restriction \tau$ is not in $j(B)$, then it must be that $\tau\not \in E$ by the same reasoning.

\end{remark}

\begin{definition}
We say the \emph{Strong Club Guessing} (SCG) holds on $S\subset \kappa$ if there exists a sequence $\la t_\alpha\subset \alpha: \alpha\in S \ra$ such that for any club $C\subset \kappa$, there exists another club $D$, satisfying that for any $\alpha\in D\cap S$, $t_\alpha$ is unbounded in $\alpha$ and $t_\alpha \subset^* C$ (namely, there exists some $\gamma<\alpha$ such that $t_\alpha-\gamma \subset C$).
\end{definition}

By taking closure, we could also assume $t_\alpha\subset \alpha$ is closed for each $\alpha<\kappa$.

\begin{proposition}\label{SCG}
In $V^{R_{\bar{U}}}$ where $\kappa$ is regular, SCG holds on the set of singular cardinals below $\kappa$.
\end{proposition}

\begin{lemma}\label{SingularFastClub}
Let $\bar{U}$ be a measure sequence without repeat points and $cf(lh(\bar{U}))\leq \kappa$. Then in $V[G]$, there exists a cofinal subset $c$ of $\kappa$ such that for any $\bar{U}$-tail-measure-one set $B\in V$, $c\subset^* O(B\cap \MS_G)$.
\end{lemma}

\begin{proof}
Assume $cf(lh(\bar{U}))=\lambda\leq \kappa$. Let $\la \nu_i: i<\lambda \ra$ be cofinal in $lh(\bar{U})$ such that for each $i<\lambda$, there is some $A_i\subset V_\kappa \cap \MS$ such that $A_i\in U(\xi)$ iff $\xi>\nu_i$. By \cite{MR2768695}, we know in $V[G]$, $cf(\kappa)=\nu<\kappa$, as witnessed by some $\la \kappa_n: n\in \nu\ra$. Recursively define $\kappa_n'$ such that $\kappa'_n$ the least element $\geq \kappa_n$ such that there is some $\bar{w}\in \MS_G\cap \bigcap_{i<\lambda, \nu_i<\kappa_n} A_i$ with $\kappa_n'=\kappa(\bar{w})$. Hence, the sequence $\la \kappa_n': n\in \nu\ra$ has the property that it is almost contained in $O(\MS_G\cap A_i)$ for any $i<\lambda$. Suppose $B$ is any other $\bar{U}$-tail-measure-one set, then there exists some large enough $i<\lambda$ such that $A_i\cap B^c$ is $\bar{U}$-null. As a result, in $V[G]$, a tail of $\MS_G$ avoids $A_i\cap B^c$. Therefore, $A_i\cap \MS_G \subset^* B\cap \MS_G$. Therefore, $\la \kappa_n': n\in \nu\ra$ is almost contained in $O(B\cap \MS_G)$.
\end{proof}

\begin{proof}[Proof of Proposition \ref{SCG}]
Let $G\subset R_{\bar{U}}$ be generic over $V$. In $V[G]$, if $\bar{w}\in G$ is such that $cf(\kappa(\bar{w}))<\kappa(\bar{w})$, then we let $c_{\kappa(\bar{w})}$ be the set as given by Lemma \ref{SingularFastClub} applied to $R_{\bar{w}}$. We claim this sequence is strong club guessing on the set of singulars below $\kappa$. Given any club $D\subset \kappa$, by Theorem \ref{tailclub} we may find a tail measure one set $B$, witnessed by $\gamma<lh(\bar{U})$ such that $O(B\cap \MS_G)\subset D$. For any $\bar{w}$ such that $B\cap V_{\kappa(\bar{w})}$ is tail measure one with respect to $\bar{w}$, Lemma \ref{SingularFastClub} implies that $c_{\kappa(\bar{w})}\subset^* O(B\cap \MS_G)$. Since $\{ \bar{w}: B\cap V_{\kappa(\bar{w})}$ is tail measure one with respect to $\bar{w} \}$ is tail measure one with respect to $\bar{U}$, we can apply Theorem \ref{tailclubconverse} to get the conclusion as desired.
\end{proof}

Even though there are many instances of club guessing principles valid in ZFC (see \cite{MR1318912}), SCG at singulars does not necessarily hold in ZFC. For example, in the forcing extension where we add $\lambda^+$ many Cohen subsets of $\lambda$, SCG on the singulars below $\lambda$ fails. We describe another scenario of a different nature.

\begin{proposition}\label{aboveStrongCompact}
Suppose $\kappa$ is a strongly compact cardinal and $\lambda>\kappa$ is a regular cardinal, then SCG fails on the set $S=\{\alpha\in \lambda: cf(\alpha)<|\alpha|\}$.
\end{proposition}

\begin{proof}
Suppose for the sake of contradiction that $\bar{t}=\langle t_\alpha: \alpha\in S\rangle$ is a strong club guessing sequence. We may assume $t_\alpha$ is a club in $\alpha$ for each $\alpha\in S$. Let $j: V\to M$ witness that $\kappa$ is $\lambda$-compact. In particular, ${}^\kappa M \subset M$ and $\sup j'' \lambda =_{def} \delta < j(\lambda)$. To see that $\delta\in j(S)$, since $j$ witnesses that $\kappa$ is $\lambda$-compact, there exists some $A\in M$ such that $j''\lambda\subset A$ and $M\models |A|<j(\kappa)<\delta$. In particular, $M\models \delta$ is singular. Hence $\delta\in j(S)$. Let $t=j(\bar{t})_\delta$ and $A=j^{-1} acc(t\cap j''\lambda)$. Recall for a set $Z$, $acc(Z)=\{\delta\in Z: Z\cap \delta \text{ is unbounded in }\delta\}$. It is not hard to see the set $A$ is $<\kappa$-closed and unbounded. Consider $\bar{A}$, the closure of $A$. By the hypothesis, there exists a club $D\subset \lambda$, such that for any $\alpha\in D \cap S$, $t_\alpha\subset^* \bar{A}$. By elementarity of $j$, we know that $\delta\in j(D \cap S)$. 
%Observe that $\delta$ is a singular cardinal in $M$, so $\delta\in j(S)$.
As a result, $t\subset^* j(\bar{A})$. This in particular implies $t\cap j''\lambda \subset^*  j(\bar{A})\cap j'' \lambda = j'' \bar{A}$.

We check that for any $\nu\in t\cap j''\lambda$, there is some $\nu'\geq \nu$ in $t\cap j'' \lambda$ that does not belong to $j'' \bar{A}$. This clearly gives a contradiction as desired. Given such a $\nu$, if it already does not belong to $j''\bar{A}$, then we are done. Suppose there is some $\xi\in \bar{A}$, $j(\xi)=\nu$. Let $\nu'$ be $\min \left( (t\cap j''\lambda) - (\nu+1)\right)$. In particular, $\nu'$ is in $nacc(t\cap j''\lambda)$. Observe that $\nu'\not\in j'' \bar{A}$. To see this, suppose for contradiction that there is $\xi'\in \bar{A}$ such that $j(\xi')=\nu'$. Note that $\xi'>\xi$. By the definition of $A$, $\xi'\in \bar{A} - A$. We can then find $\eta\in (\xi, \xi')\cap A$. But then $\nu<j(\eta)<\nu'$ and $j(\eta)\in t\cap j''\lambda$, contradicting with the choice of $\nu'$.
\end{proof}

\begin{remark}
Proposition \ref{aboveStrongCompact} can be used to separate the $\diamondsuit^+$-principles from the SCG principles. For example, \cite[Question 1]{MR2194236} asks if for a regular uncountable cardinal $\mu$, $\diamondsuit^+(\mu)$ implies SCG on $\mu$. By \cite[Theorem 22]{MR1838355}, if $2^\lambda=\lambda^+$ and $2^{\lambda^+}=\lambda^{++}$, then there is a $\lambda^+$-directed closed and $\lambda^{++}$-c.c poset that forces $\diamondsuit^+(\lambda^+)$. Therefore, relative to the existence of a supercompact cardinal $\kappa$, we can produce a model (using the technique by Laver \cite{MR0472529}) where $\kappa$ remains supercompact and $\lambda>\kappa$ satisfies $\diamondsuit^+(\lambda^+)$ holds. In this model, SCG on $\lambda^+$ fails by Proposition \ref{aboveStrongCompact}. There are two things worth noting: 
    \begin{enumerate}
        \item Large cardinals should not be needed to separate $\diamondsuit^+$ from SCG. For example, \cite{MR2963019} shows it is consistent relative to ZFC that $\diamondsuit^+(\omega_1)$ holds and SCG on $\omega_1$ fails. 
        \item If $\lambda$ is regular with $\lambda^{<\lambda}=\lambda$, then \cite{MR2151585} proves there is a $<\lambda$-closed and $\lambda^+$-c.c poset that adds a SCG-sequence on $\lambda$.
    \end{enumerate}
\end{remark}

The validity of SCG on regulars is constrained by the compactness principle to be discussed in the next section.

\section{A weak consequence of ineffability does not imply the diamond principle}\label{amenable}
To relate the compactness principle we will consider in this section with standard large cardinals, we first supply a characterization of ineffable cardinals in terms of club sequences. The following lemma is essentially due to Todorcevic \cite{MR2355670}, who proved the corresponding version for weakly compact cardinals. Recall $\kappa$ is \emph{ineffable} iff for any $\la d_\alpha\subset \alpha: \alpha<\kappa\ra$, there exists $A\subset \kappa$ such that $\{\alpha<\kappa: A\cap \alpha=d_\alpha\}$ is a stationary subset of $\kappa$.

\begin{definition}
A $C$-sequence on $\kappa$ is a sequence $\langle C_\alpha: \alpha<\kappa\rangle$ where each $C_\alpha$ is a club subset of $\alpha$ for any $\alpha<\kappa$.
\end{definition}

\begin{lemma}
For a strongly inaccessible $\kappa$, we have that: 
$\kappa$ is ineffable iff for any $C$-sequence on $\kappa$ $\langle C_\alpha\subset \alpha: \alpha<\kappa\ra$, there exists a club $D\subset \kappa$, satisfying that $\{\alpha<\kappa: D\cap \alpha = C_\alpha\}$ is stationary. 
\end{lemma}

\begin{proof}
We only prove the non-trivial ``if'' direction. Namely, suppose every $C$-sequence on inaccessible $\kappa$ admits a club $D\subset \kappa$ such that $\{\alpha<\kappa: D\cap \alpha =C_\alpha\}$ is stationary, then $\kappa$ is ineffable.

Let $T$ be a tree of height $\kappa$ and the $\xi$-th level of the tree is $T(\xi)=[\delta_\xi, \delta_{\xi+1})$, where $\langle \delta_\xi: \xi<\kappa\rangle$ is an increasing continuous sequence of ordinals below $\kappa$. We will also assume that $T$ \emph{does not split at limit levels}, in the sense that for any limit $\xi<\kappa$, and $\alpha_0, \alpha_1\in T(\xi)$, if $\{\eta: \eta<_T \alpha_0\}=\{\eta: \eta<_T \alpha_1\}$, then $\alpha_0=\alpha_1$.

We claim that there is a branch $d$ through $T$ such that $\{\xi<\kappa: \delta_\xi\in d\}$ is stationary in $\kappa$. Note that there is a club $C$ consisting limit $\xi$ such that $\delta_\xi =\xi$. Consider the following $C$-sequence $\bar{C}$: for any $\xi<\kappa$, if $\xi\in \lim C$, then we let $C_\xi$ be the closure of $E_\xi=_{def}\{\eta<\xi: \eta<_T \xi\}$ and if $\xi\not\in \lim C$, then $C_\xi = (\max C\cap \xi, \xi)$. Recall the hypothesis stated at the beginning of the proof. We can then find a club $D\subset \kappa$ such that $S=\{\xi<\kappa: D\cap \xi = C_\xi\}$ is stationary in $\kappa$. We claim that $S\cap \lim C$ forms a branch as desired. Suppose $\alpha<\beta \in S\cap \lim C$, let $\alpha'\in T(\alpha)$ be the unique node such that $\alpha'<_{T} \beta$. We need to show that $\alpha=\alpha'$. By the fact that $T$ does not split at limit levels, we only need to show that $\alpha$ and $\alpha'$ have the same set of $T$-predecessors. Because $T$ is a tree, it suffices to show there are arbitrarily large $\eta<\alpha=\delta_\alpha$ with $\eta<_T \alpha, \alpha'$. Given $\xi<\alpha$, note that $D\cap T(\xi+1)$ has a unique element. To see this, by the hypothesis we have $D\cap \alpha = C_\alpha$, as a result $D\cap T(\xi+1)$ consists of the unique node that is $T$-below $\alpha$, since $D\cap T(\xi+1)=C_\alpha\cap T(\xi+1)=E_\alpha \cap T(\xi+1)$. Let $\eta\in D\cap T(\xi+1)$. The same argument shows that $\eta<_T \alpha'$. We are done.

It is left to see the principle as above implies $\kappa$ is ineffable. Given a list $\langle d_\alpha \in 2^{\alpha}: \alpha<\kappa\rangle$, consider the tree $T'=\{d_\beta\restriction \alpha: \alpha\leq \beta\}$. Note $T'$ does not split on limit levels. Since $\kappa$ is strongly inaccessible, it is possible to find an increasing continuous sequence $\langle \delta_\xi: \xi<\kappa\rangle$ such that there is a bijection $b_\xi: [\delta_\xi, \delta_{\xi+1})\leftrightarrow T'(\xi)$ with $b_\xi(\delta_\xi)= d_\xi$ for each $\xi<\kappa$. Let $T$ be the pull-back of $T'$, namely, for any $\alpha_0, \alpha_1$ with $\xi_0, \xi_1< \kappa $ be such that $\alpha_i\in [\delta_{\xi_i},\delta_{\xi_i+1})$ for $i<2$, $\alpha_0 <_T \alpha_1$ iff $b_{\xi_0}(\alpha_0)<_{T'} b_{\xi_1}(\alpha_1)$. It can be easily verified that $T$ is of the form described above. Then there is a branch $d$ through $T$ such that $S=\{\xi<\kappa: \delta_\xi\in d\}$ is stationary. Consider $d'=\{b_\xi(\delta_\xi)=_{def} d_\xi : \xi\in S\}$. Then for any $\xi_0<\xi_1\in S$, $\delta_{\xi_0}<_T \delta_{\xi_1}$, which implies $b_{\xi_0}(\delta_{\xi_0})=d_{\xi_0}<_{T'} b_{\xi_1}(\delta_{\xi_1})=d_{\xi_1}$, hence $d_{\xi_0}=d_{\xi_1}\restriction \xi_0$. This implies that $\kappa$ is ineffable.

\end{proof}

The following notion is related to constructions of distributive Aronszajn trees, which strenghtens the principle $\otimes_{\overrightarrow{C}}$ from \cite[p. 134]{MR1318912}.

\begin{definition}[\cite{MR3914943}]
A $C$-sequence on $\kappa$ $\langle C_\alpha: \alpha<\kappa\ra$ is said to be \emph{amenable} if there is no club $D\subset \kappa$, such that $\{\alpha<\kappa: D\cap \alpha \subset C_\alpha\}$ is stationary.
\end{definition}

Therefore, the assertion that \emph{there is no amenable C-sequence on $\kappa$} is a weakening of ineffability. To be more explicit, the compactness principle states: 

\emph{for any $C$-sequence $\la C_\alpha\subset \alpha: \alpha<\kappa\ra$, there exists a club $D\subset \kappa$, such that $\{\alpha: D\cap \alpha \subset C_\alpha\}$ is a stationary subset of $\kappa$}.

\begin{remark}
That ``there is no amenable C-sequence on $\kappa$'' is a non-trivial compactness principle. For example, it implies that the \emph{diagonal stationary reflection principle at $\kappa$}, which means: for any $\kappa$-sequence of stationary subsets of $\kappa$ $\la S_i: i<\kappa\ra$, there is some $\delta<\kappa$ such that for any $i<\delta$, $S_i\cap \delta$ is stationary. However, it is not a consequence of $\kappa$ being weakly compact, while the diagonal stationary reflection principle at $\kappa$ is. For instance, in $L$, there is no amenable C-sequence on $\kappa$ iff $\kappa$ is ineffable.
\end{remark}

The following seems open: 
\begin{question}
Is ``there exists a cardinal carrying no amenable C-sequence'' equiconsistent with ``there exists an ineffable cardinal''?
\end{question}

\begin{remark}
To relate to the Strong Club Guessing principle discussed in the previous section, notice that if there is no amenable C-sequence on $\kappa$, then the Strong Club Guessing principle on the regulars below $\kappa$ fails.
\end{remark}

In \cite{MR3960897}, it was shown that relative to the existence of large cardinals, it is consistent that there is a strongly inaccessible cardinal $\kappa$ where the diagonal stationary reflection principle holds yet $\diamondsuit(\kappa)$ fails. We strengthen the compactness principle in this model to ``there is no amenable C-sequence on $\kappa$''.

\begin{theorem}\label{stationarilytrivial}
If $\bar{U}$ is a measure sequence on $\kappa$ such that $cf(lh(\bar{U}))\geq \kappa^{++}$, then in $V^{R_{\bar{U}}}$, there is no amenable C-sequence on $\kappa$. 
\end{theorem}

\begin{lemma}\label{pseudo}
Suppose $cf(lh(\bar{U}))\geq \kappa^+$. Let $E\subset \mathcal{MS}$ be $\bar{U}$-stationarily positive, namely $\{\tau<lh(\bar{U}): E\in U(\tau)\}$ is a stationary subset of $lh(\bar{U})$.
Given a tail measure function $b$, there exists $A\in \bigcap \bar{U}\restriction \geq \beta$ for some $\beta<lh(\bar{U})$, such that in any generic $G\subset R_{\bar{U}}$, $$\{\alpha<\kappa: \exists \bar{w}\in \MS_G\cap E [\kappa(\bar{w})=\alpha \wedge O(A\cap \MS_G)\cap \alpha\subset^* O(b(\bar{w})\cap \MS_G)]\}$$ is stationary.
\end{lemma}

\begin{proof}
%Notice that $\{\bar{w}: \bar{w} \text{ does not contain a weak repeat point}\}\in \bigcap \bar{U}$. So we may work inside this set.
Fix a tail measure function $b$. Then, for every $\bar{w}\in \MS \cap V_\kappa$, there is some $\tau=\tau_{\bar{w}}<lh(\bar{w})$ such that $b(\bar{w})\in \bigcap \bar{w}\restriction \geq \tau$. Let $f$ be the function $\bar{w}\mapsto \tau_{\bar{w}}$. Let $S=\{\tau<lh(\bar{U}): E\in U(\tau)\}$. Consider $g: S\to lh(\bar{U})$ such that $g(\gamma)=j(f)(\bar{U}\restriction \gamma)<\gamma$. By Pressing Down, there is a stationary $S'\subset S$ and $\beta<lh(\bar{U})$ such that $g'' S' \subset \beta$. We may assume $\beta$ is novel since $\bar{U}$ does not contain a repeat point. Pick $A_\beta\in \bigcap \bar{U}\upharpoonright\geq \beta$ with $A_\beta^c\in \bigcap \bar{U}\restriction \beta$.

We check that $F=\{\bar{w}\in E: A_\beta\cap b(\bar{w})^c \text{ is }\bar{w}\text{-null}\}$ is $\bar{U}$-positive. Indeed, for any $\tau\in S'$, $\bar{U}\restriction \tau \in j(F)$, since $A_\beta \cap j(b)(\bar{U}\restriction \tau)^c\not\in U(i)$ for any $i<\tau$. To see this, if $i\geq \beta$, $j(b)(\bar{U}\restriction \tau)\in U(i)$ and if $i<\beta$, $A^c_\beta\in U(i)$.

Let $G\subset R_{\bar{U}}$ be generic over $V$, then $O(F\cap \MS_G)$ is a stationary subset of $\kappa$, by Theorem \ref{tailclub}. For each $\bar{w}\in F\cap \MS_G$, we know a tail of $\MS_G$ below $\bar{w}$ avoids $A_\beta \cap b(\bar{w})^c$. As a result, $O(A_\beta\cap \MS_G)\cap \kappa(\bar{w})\subset^* O(b(\bar{w})\cap \MS_G)$.

\end{proof}

\begin{lemma}\label{reducetoground}
Suppose $cf(lh(\bar{U}))\geq \kappa^+$. Let $p\Vdash \dot{c}$ is a name for a $C$-sequence. Then there exists $q\leq^* p$ and a tail measure function $b$ such that for some $B\in \bigcap \bar{U}$, for any $\bar{w}\in B$, if $cf(lh(\bar{w}))\geq \kappa(\bar{w})^+$, it is the case that $q^\frown \bar{w}\Vdash O(b(\bar{w})\cap \MS_G)\subset^* \dot{c}(\bar{w})$.
\end{lemma}

\begin{proof}
Suppose $p=p_0^\frown (\bar{U}, A)$. By Lemma \ref{RepresentationLemma0}, we can find $\bar{U}$-measure-one  $A'\subset A$, and some $\dot{d}$ such that $\dot{d}(\bar{w})$ is a $R_{\bar{w}}$-name for a club subset of $\kappa(\bar{w})$ and letting $q=p_0^\frown (\bar{U}, A')$ we have that $q^\frown \bar{w}\Vdash \dot{c}_{\kappa(\bar{w})}=\dot{d}(\bar{w})$ for any $\bar{w}\in A'$.

Let $R_{<\kappa}$ denote the collection of possible lower parts of conditions in $R_{\bar{U}}$. Namely, $t\in R_{<\kappa}$ iff $t\fr (\bar{U}, E)\in R_{\bar{U}}$ for some $E\in \bigcap \bar{U}$.
For each $\bar{w}\in A'$ with $cf(lh(\bar{w}))\geq \kappa(\bar{w})^+$, $t\in V_{\kappa(\bar{w})}\cap R_{<\kappa}$, we can find (by Theorem \ref{tailclub}) $A^{t}_{\bar{w}}\in \bigcap\bar{w}$ and some $\bar{w}$-tail-measure-one set $\Gamma_{\bar{w}}^t $ such that $t^\frown (\bar{w}, A^t_{\bar{w}})\Vdash_{R_{\bar{w}}} O(\Gamma_{\bar{w}}^t\cap \MS_{\dot{G}_{\bar{w}}})\subset \dot{d}(\bar{w})$. Let $A_{\bar{w}}=\Delta_{t} A_{\bar{w}}^t \in \bigcap \bar{w}$ and $\Gamma_{\bar{w}}=\Delta_{t} \Gamma^t_{\bar{w}}$. The fact that $\Gamma_{\bar{w}}$ is a $\bar{w}$-tail-measure-one set follows from the fact that $cf(lh(\bar{w}))\geq \kappa(\bar{w})^+$.

Let $b$ be a tail measure function satisfying that $b(\bar{w})=\Gamma_{\bar{w}}$ for any $\bar{w}\in A'$ with $cf(lh(\bar{w}))\geq \kappa(\bar{w})^+$. We claim that $q^\frown \bar{w}\Vdash O(b(\bar{w})\cap \MS_G)\subset^* \dot{c}(\bar{w})$ for all $\bar{w}\in A'$ with $cf(lh(\bar{w}))\geq \kappa(\bar{w})^+$. Fix such $\bar{w}$.

It is left to check that $p_0^\frown (\bar{w}, A'\cap V_{\kappa(\bar{w})})\Vdash_{R_{\bar{w}}} O(\Gamma_{\bar{w}}\cap \MS_{G_{\bar{w}}})\subset^* \dot{d}(\bar{w})$. This follows immediately from the fact that for any $t\in V_{\kappa(\bar{w})}$ with $t\fr (\bar{w}, A_{\bar{w}})\in R_{\bar{w}}$, $t^\frown (\bar{w}, A_{\bar{w}})\Vdash_{R_{\bar{w}}} O(\Gamma_{\bar{w}}\cap \MS_{\dot{G}_{\bar{w}}})- (rank(t)+1) \subset \dot{d}(\bar{w})$.
To see this fact, suppose for the sake of contradiction that this is not the case for some $t$, then we have an extension $t^*=t'^\frown (\bar{u}, e)^\frown (\overrightarrow{\eta}, \overrightarrow{A}_{\overrightarrow{\eta}})^\frown (\bar{w}, A^*)\Vdash \kappa(\bar{u})\not \in \dot{d}(\bar{w})$ for some $\bar{u}\in \Gamma_{\bar{w}}$. Notice that $\bar{u}\in A_{\bar{w}}^{t'}$ and $\overrightarrow{\eta}\subset A_{\bar{w}}^{t'}$ by the definition of $A_{\bar{w}}$. Hence $t^*$ is compatible with ${t'}^\frown (\bar{w}, A_{\bar{w}}^{t'})$, which forces $O(\Gamma_{\bar{w}}^{t'}\cap \MS_{\dot{G}_{\bar{w}}})\subset \dot{d}(\bar{w})$. As $\bar{u}\in \Gamma_{\bar{w}}^{t'}$, we know that ${t'}^\frown (\bar{u}, A_{\bar{w}}^{t'}\cap V_{\kappa(\bar{u})}) ^\frown (\bar{w}, A_{\bar{w}}^{t'}) \Vdash \kappa(\bar{u})\in \dot{d}(\bar{w})$, and is compatible with $t^*$ that forces $\kappa(\bar{u})\not\in \dot{d}(\bar{w})$, which is absurd.

\end{proof}

\begin{proof}[Proof of Theorem \ref{stationarilytrivial}]
In $V[G]$, where $\langle c_\alpha: \alpha<\kappa\rangle$ is given, we aim to find a club $C\subset \kappa$ such that $\{\alpha: C\cap \alpha\subset^* c_\alpha\}$ is a stationary subset of $\kappa$. By Lemma \ref{reducetoground}, we can find a tail measure function $b\in V$, such that on a tail of $\MS_G$, if $\bar{w}\in \MS_G$ satisfies $cf(lh(\bar{w}))\geq \kappa(\bar{w})^+$, then $O(b(\bar{w})\cap \MS_G)\subset^* c_{\kappa(\bar{w})}$. By Lemma \ref{pseudo} with $E=\{\bar{w}: cf(lh(\bar{w}))\geq \kappa(\bar{w})^+\}$, we can find a club $D$ in $V[G]$ such that for stationarily many regular $\alpha\in C_G$ with $\alpha=\kappa(\bar{w})$ for $\bar{w}\in \MS_G$, it is the case that $D\cap \alpha\subset^* O(b(\bar{w})\cap \MS_G) \subset^* c_{\alpha}$.
\end{proof}

If we prepare the ground model such that $\kappa$ is $(\kappa+2)$-strong and $2^{\kappa}>\kappa^{++}$, then for some elementary embedding, we can derive a measure sequence $\bar{U}$ of length $\kappa^{++}$. By Theorem \ref{diamondfailsWoodin} and Theorem \ref{stationarilytrivial}, we have no amenable C-sequence on $\kappa$ and $\neg \diamondsuit(\kappa)$ in $V^{R_{\bar{U}}}$. 

\begin{corollary}
Relative to the existence of suitable large cardinals, it is consistent that there does not exist an amenable C-sequence at $\kappa$ and $ \diamondsuit(\kappa)$ fails.
\end{corollary}

\section{Questions and Acknowledgments}
The question of whether $\kappa$ being weakly compact implies $\diamondsuit(\kappa)$ is still open and we know the method of Radin forcing cannot provide a solution. However, a better understanding of compactness principles in the Radin model can still shed light on the relationship between weak compactness and its weakenings. 

\begin{question}\label{question}
For strongly inaccessible $\kappa$, can we produce a model where $\kappa$ is not weakly compact and either of the following holds: 
\begin{enumerate}
    \item $\kappa$ is Jonsson,
    \item there is no $\kappa$-Suslin tree,
    \item $\kappa \to [\kappa]^2_\omega$.
\end{enumerate}
\end{question}

As a first step towards the solution, we may ask if we can characterize the measure sequence $\bar{U}$ such that $\kappa$ satisfies any of the compactness principles mentioned in Question \ref{question} in $V^{R_{\bar{U}}}$.\\

${}$\\
\textbf{Acknowledgments:} The authors would like to thank Tom Benhamou, Moti Gitik and Assaf Rinot for valuable questions and comments. The authors also want to thank the anonymous referee for a careful reading and for their very helpful corrections, comments and suggestions.

\bibliographystyle{alpha}
\bibliography{bib}

\end{document}

%% file: OmerMakros.tex
\newcommand{\crit}{\mathop{\mathrm{crit}}}

\newcommand{\fr}{{}^\frown}

\newcommand{\la}{\langle}

\newcommand{\ra}{\rangle}

\newcommand{\uhr}{\restriction}

\newcommand{\ol}{\ol}

\renewcommand{\ol}{\bar}

\newcommand{\MS}{\mathcal{MS}}

\DeclareMathOperator{\dom}{dom}

%% file: Radin.bbl
\begin{thebibliography}{CFM01}

\bibitem[BN19]{MR3960897}
Omer Ben-Neria.
\newblock Diamonds, compactness, and measure sequences.
\newblock {\em J. Math. Log.}, 19(1):1950002, 20, 2019.

\bibitem[BR19]{MR3914943}
Ari~Meir Brodsky and Assaf Rinot.
\newblock Distributive {A}ronszajn trees.
\newblock {\em Fund. Math.}, 245(3):217--291, 2019.

\bibitem[CFM01]{MR1838355}
James Cummings, Matthew Foreman, and Menachem Magidor.
\newblock Squares, scales and stationary reflection.
\newblock {\em J. Math. Log.}, 1(1):35--98, 2001.

\bibitem[Cum]{CummingsWoodin}
James Cummings.
\newblock Woodin's theorem on killing diamond via {R}adin forcing.
\newblock {\em Unpublished note}.

\bibitem[Cum13]{CummingsMagidor}
James Cummings.
\newblock {\em private communication}, 14 July 2013.

\bibitem[CW]{CummingsWoodinbook}
James Cummings and Hugh Woodin.
\newblock {\em Generalized Prikry forcings}.
\newblock Unpublished book.

\bibitem[DH06]{MR2279655}
Mirna D\v{z}amonja and Joel~David Hamkins.
\newblock Diamond (on the regulars) can fail at any strongly unfoldable
  cardinal.
\newblock {\em Ann. Pure Appl. Logic}, 144(1-3):83--95, 2006.

\bibitem[FK05]{MR2151585}
Matthew Foreman and Peter Komjath.
\newblock The club guessing ideal: commentary on a theorem of {G}itik and
  {S}helah.
\newblock {\em J. Math. Log.}, 5(1):99--147, 2005.

\bibitem[Git10]{MR2768695}
Moti Gitik.
\newblock Prikry-type forcings.
\newblock In {\em Handbook of set theory. {V}ols. 1, 2, 3}, pages 1351--1447.
  Springer, Dordrecht, 2010.

\bibitem[Hau92]{MR1164732}
Kai Hauser.
\newblock Indescribable cardinals without diamonds.
\newblock {\em Arch. Math. Logic}, 31(5):373--383, 1992.

\bibitem[IL12]{MR2963019}
Tetsuya Ishiu and Paul~B. Larson.
\newblock Some results about {$(+)$} proved by iterated forcing.
\newblock {\em J. Symbolic Logic}, 77(2):515--531, 2012.

\bibitem[Ish05]{MR2194236}
Tetsuya Ishiu.
\newblock Club guessing sequences and filters.
\newblock {\em J. Symbolic Logic}, 70(4):1037--1071, 2005.

\bibitem[Jen91]{Jensen69}
Ronald Jensen.
\newblock Notes from {O}berwolfach.
\newblock {\em Unpublished manuscript}, 1991.

\bibitem[JK69]{JensenKunen}
Ronald Jensen and Kenneth Kunen.
\newblock Some combinatorial properties of {L} and {V}.
\newblock {\em Unpublished manuscript}, 1969.

\bibitem[Kun11]{MR2905394}
Kenneth Kunen.
\newblock {\em Set theory}, volume~34 of {\em Studies in Logic (London)}.
\newblock College Publications, London, 2011.

\bibitem[Lav78]{MR0472529}
Richard Laver.
\newblock Making the supercompactness of {$\kappa $} indestructible under
  {$\kappa $}-directed closed forcing.
\newblock {\em Israel J. Math.}, 29(4):385--388, 1978.

\bibitem[Mit82]{Mitchell82}
William Mitchell.
\newblock How weak is a closed unbounded ultrafilter?
\newblock In D.~{Van Dalen}, D.~Lascar, and T.J. Smiley, editors, {\em Logic
  Colloquium '80}, volume 108 of {\em Studies in Logic and the Foundations of
  Mathematics}, pages 209--230. Elsevier, 1982.

\bibitem[Rad82]{MR670992}
Lon~Berk Radin.
\newblock Adding closed cofinal sequences to large cardinals.
\newblock {\em Ann. Math. Logic}, 22(3):243--261, 1982.

\bibitem[She94]{MR1318912}
Saharon Shelah.
\newblock {\em Cardinal arithmetic}, volume~29 of {\em Oxford Logic Guides}.
\newblock The Clarendon Press, Oxford University Press, New York, 1994.
\newblock Oxford Science Publications.

\bibitem[Tod07]{MR2355670}
Stevo Todorcevic.
\newblock {\em Walks on ordinals and their characteristics}, volume 263 of {\em
  Progress in Mathematics}.
\newblock Birkh\"auser Verlag, Basel, 2007.

\bibitem[Zem00]{Zeman00}
Martin Zeman.
\newblock {$\diamondsuit$} at {M}ahlo cardinals.
\newblock {\em J. Symbolic Logic}, 65(4):1813--1822, 2000.

\end{thebibliography}
